\documentclass{article}
\usepackage{graphicx} 

\title{Monomial stability of Frobenius images}
\author{Nikita Borisov}
\date{\today}
\usepackage[margin=1in]{geometry}
\usepackage{soul}
\usepackage{graphicx}
\usepackage{cancel}
\usepackage{amsmath}
\usepackage{amssymb}
\usepackage{amsthm}
\usepackage{subfig}
\usepackage{makecell}
\usepackage{hyperref}
\usepackage{ytableau}
\hypersetup{
    colorlinks=true,
    linkcolor=blue,
    filecolor=magenta,      
    urlcolor=blue,
    citecolor=blue,
    pdftitle={Overleaf Example},
    pdfpagemode=FullScreen,
    }
\usepackage{listings}
\usepackage{enumitem}
\usepackage{tikz}

\usepackage[nonamebreak,square]{natbib}
\setcitestyle{numbers, open={[}, close = {]}}
\graphicspath{ {./images/} }

\usepackage[OT2,T1]{fontenc}
\DeclareSymbolFont{cyrletters}{OT2}{wncyr}{m}{n}
\DeclareMathSymbol{\sh}{\mathalpha}{cyrletters}{"78}

\begin{document}

\newcommand{\R}[0]{\mathbb{R}}
\newcommand{\C}[0]{\mathbb{C}}
\renewcommand{\H}[0]{\mathbb{H}}
\renewcommand{\S}[0]{\mathbb{S}}
\newcommand{\Z}[0]{\mathbb{Z}}
\newcommand{\Q}[0]{\mathbb{Q}}
\newcommand{\N}[0]{\mathbb{N}}
\newcommand{\F}[0]{\mathbb{F}}
\newcommand{\p}[1]{\left(#1\right)}
\newcommand{\id}[0]{\text{id}}
\renewcommand{\phi}[0]{\varphi}
\newcommand{\im}[0]{\text{im}\,}
\renewcommand{\ker}[0]{\text{ker}\,}
\newcommand{\coker}[0]{\text{coker}\,}
\newcommand{\cis}[0]{\text{cis}}
\renewcommand{\choose}[2]{\p{\begin{matrix}#1\\ #2\end{matrix}}}
\newcommand{\schoose}[2]{\p{\begin{smallmatrix}#1\\ #2\end{smallmatrix}}}
\newcommand{\sgn}[0]{\text{sgn}}
\newcommand{\bchoose}[2]{\begin{bmatrix}#1\\#2\end{bmatrix}}
\renewcommand{\char}[0]{\text{char}}
\newcommand{\Aut}[0]{\text{Aut}}
\renewcommand{\o}[1]{\overline{#1}}
\renewcommand{\P}{\mathbb{P}}
\newcommand{\wt}{\text{wt}}
\newcommand{\rg}{\text{rg}}
\newcommand{\inv}{\text{inv}}
\newcommand{\maj}{\text{maj}}
\newcommand\parspace[1][\bigskip]{\par#1\noindent\ignorespaces}

\newtheorem{prop}{Proposition}[subsection]
\newtheorem{lem}{Lemma}[subsection]
\newtheorem{thm}{Theorem}[subsection]
\newtheorem{conj}{Conjecture}[subsection]
\newtheorem{cor}{Corollary}[subsection]

\theoremstyle{definition}
\newtheorem{defi/}{Definition}[subsection]
\newtheorem{ques/}{Problem}[subsection]
\newtheorem{ex/}{Example}[subsection]
\newtheorem{rem/}{Remark}[subsection]

\newenvironment{defi}
  {\renewcommand{\qedsymbol}{|}%
   \pushQED{\qed}\begin{defi/}}
  {\popQED\end{defi/}}
\newenvironment{ex}
  {\renewcommand{\qedsymbol}{$\triangle$}%
   \pushQED{\qed}\begin{ex/}}
  {\popQED\end{ex/}}
\newenvironment{rem}
  {\renewcommand{\qedsymbol}{$\triangle$}%
   \pushQED{\qed}\begin{rem/}}
{\popQED\end{rem/}}
\newenvironment{ques}
  {\renewcommand{\qedsymbol}{}%
   \pushQED{\qed}\begin{ques/}}
{\popQED\end{ques/}}

\maketitle

\begin{abstract}
    We study representation stability in the sense of Church, Ellenberg, and Farb \cite{FI-module} through the lens of symmetric function theory and the different symmetric function bases. We show that a sequence, $(F_n)_n$, where $F_n$ is a homogeneous symmetric function of degree $n$, has stabilizing Schur coefficients if and only if it has stabilizing monomial coefficients. More generally, we develop a framework for checking when stabilizing coefficients transfer from one symmetric function basis to another. We also see how one may compute representation stable ranges from the monomial expansions of the $F_n$.\parspace

    As applications, we reprove and refine the representation stability of diagonal coinvariant algebras, $DR_n$. We also observe new representation stability phenomena of the Garsia-Haiman modules. This establishes certain stability properties of the modified Macdonald polynomials, $\tilde{H}_{\mu^{(n)}}[X;q,t]$ and the modified $q,t$-Kostka numbers, $\tilde{K}_{\mu^{(n)},\nu[n]}(q,t)$, for arbitrary sequences of partitions with $\mu^{(n)}\vdash n$ and $\mu^{(n)}\subseteq \mu^{(n+1)}$.
\end{abstract}

\section{Introduction}

In their seminal work, Church, Ellenberg, and Farb \cite{FI-module} created an effective framework called representation stability for studying various previously observed phenomena. The typical example involves a naturally occurring sequence $(V_n)_{n\in \N}$, where each $V_n$ is an $A_n$-module for an algebra $A_n$, typically the group algebra of the symmetric group $\textbf{k}[S_n]$ over some field $\textbf{k}$. Despite the fact that different $n$ yield representations of different algebras, these sequences may have decompositions into irreducibles that surprisingly ``stabilize'' in some appropriate sense. As an example, consider the configuration space of $n$ distinct points in $\C$, $\text{Conf}_n(\C)$. Then the natural action of $S_n$ on this space induces an action on the cohomology $H^k(\text{Conf}_n(\C))$, and hence yields an $S_n$-representation. Letting $\mathbb{S}^\lambda$ denote the irreducible representation of $S_n$ corresponding to partition $\lambda\vdash n$ and taking $k=1$, the first few representations admit the following decompositions
\begin{align*}
                    H^1\big(\text{Conf}_2(\mathbb{C})\big) &\cong \mathbb{S}^{(2)} \oplus \mathbb{S}^{(1,1)} \\
                    H^1\big(\text{Conf}_3(\mathbb{C})\big) &\cong \mathbb{S}^{(3)} \oplus \mathbb{S}^{(2,1)} \\
                    H^1\big(\text{Conf}_4(\mathbb{C})\big) &\cong \mathbb{S}^{(4)} \oplus \mathbb{S}^{(3,1)} \oplus \mathbb{S}^{(2,2)} \\
                    H^1\big(\text{Conf}_5(\mathbb{C})\big) &\cong \mathbb{S}^{(5)} \oplus \mathbb{S}^{(4,1)} \oplus \mathbb{S}^{(3,2)}\\
                    H^1\big(\text{Conf}_6(\mathbb{C})\big) &\cong \mathbb{S}^{(6)} \oplus \mathbb{S}^{(5,1)} \oplus \mathbb{S}^{(4,2)}
\end{align*}
and this pattern continues with $ H^1\big(\text{Conf}_n(\mathbb{C})\big) \cong \mathbb{S}^{(n)} \oplus \mathbb{S}^{(n-1,1)} \oplus \mathbb{S}^{(n-2,2)}$ for all $n\geq 4$. Note the multiplicity of $\mathbb{S}^{(n-|\lambda|,\lambda_1,\cdots,\lambda_{l(\lambda)})}$ has stabilized to a constant for each possible $\lambda$ and we will state the precise definition of the phenomena in the preliminaries. In the example above the stable multiplicities are 1 for $\lambda=\emptyset,(1),(2)$ and zero otherwise. Church and Farb \cite{Church_2013} also observed this phenomena with the cohomology of pure braid groups, homology of certain Torelli subgroups, and cohomology of flag varieties among many other examples.\parspace

The major contribution of Church, Ellenberg, and Farb \cite{FI-module} was the categorification of representation stability. It turns out that all the examples of $(V_n)_{n\in \N}$ mentioned have natural connective maps $\phi_n:V_n\rightarrow V_{n+1}$, which satisfy certain properties making them compatible with representations in the sequence. The authors of \cite{FI-module} identified the sequences $(V_n,\phi_n)_{n\in \N}$ with functors from the category of finite sets with injections, \textbf{FI}, to the category of $\C$-modules; such functors are called \textbf{FI}-modules. They were then able to show that a sequence exhibits representation stability if and only if the corresponding functor is finitely generated in the category of \textbf{FI}-modules.\parspace

Furthermore, they provided an explicit range $n\geq \text{stab-deg}(V_\bullet)+\wt(V_\bullet)$ after which point the multiplicities of the $\mathbb{S}^{(n-|\lambda|,\lambda_1,\cdots,\lambda_{l(\lambda)})}$ would stabilize. However, this lower bound may not be sharp and the quantities $\text{stab-deg}(V_\bullet),~\wt(V_\bullet)$ known as stability degree and weight, respectively, may be difficult to compute. One of the major goals of this paper is to provide ways of getting around needing to compute stability degree in particular.\parspace

Many authors have generalized the work of Church, Ellenberg, and Farb to other categories (\cite{Gadish_2017} \cite{gan2015noetherianpropertyinfiniteei} \cite{gan2017representationstabilitytheoremvimodules} \cite{Laudone} to name a few) as well as working on finding explicit stable decompositions and bounds for 
 different kinds of stable ranges of given sequences $(V_n)_{n\in \N}$ \cite{bahran}.  One inviting approach for combinatorialists is to use the Frobenius characteristic map, which provides a bridge between representations of $S_n$ and symmetric functions (see \cite{sagan} Section 4.7). For instance, Hersch and Reiner \cite{hershReiner} worked heavily with higher Lie characters in the algebra of symmetric functions in order to find stable ranges for the cohomology of configuration spaces of points in $\R^d$.\parspace

Another benefit of working with the Frobenius image of representation is that we may have some formula describing the symmetric function corresponding to a representation, when the irreducible decomposition or even character values may not be known. For example, the diagonal coinvariant algebras have no known combinatorial descriptions for their irreducible decompositions but their Frobenius images have an explicit monomial formula proved by Carlsson and Mellit \cite{shuffle} and first conjectured in \cite{haglund2004combinatorialformulacharacterdiagonal}. Similarly, the modified Macdonald polynomials, $\tilde{H}_\mu[X;q,t]$, which are the Frobenius images of certain $S_n$-modules called Garsia-Haiman modules, have no known combinatorial formula for their Schur expansion, but have a monomial formula given in \cite{HHL}. Both of these examples turn out to exhibit representation stability.\parspace

We describe a method for leveraging these monomial expansions to prove representation stability and provide stable ranges that may not be strict but are easy to compute from the combinatorics of the monomial expansions. One of the main observations of this paper is that a sequence of symmetric functions has stabilizing Schur expansions if and only if it has stabilizing monomial expansion, but the monomial expansions may not stabilize uniformally (Proposition \ref{genTransfer}). Leading up to this observation, we study how stability of expansions transfers between pairs of symmetric function bases in general, i.e. for bases that are not the Schur or monomial ones. In Proposition \ref{stabTransferResult}, we provide necessary and sufficient conditions on the change of basis coefficients to get stability transfer between a pair of bases. This proposition is then systematically applied to each pair of classical symmetric function bases. We identify the monomial and ``class function'' basis, $p_\lambda/z_\lambda$, as being particularly compatible with Schur functions yielding Theorem \ref{mainResult}, the main result. In the second half of the paper, we apply the theorem to find stable ranges of some diagonal coinvariant algebras and Macdonald polynomials.\parspace

Section \ref{prelims} covers the preliminaries including necessary results for torsion-free \textbf{FI}-modules and combinatorial formulas for the change of basis coefficients between the classical symmetric function bases. Section \ref{transferSec} proves the conditions for stability transfer and applies it to the classical bases, concluding with the main result. Sections \ref{coinvSec} and \ref{macSec} cover the applications to coinvariant algebras and Macdonald polynomials. It turns out that the stable range in Theorem \ref{mainResult}(a) can be slightly sharpened and proved in a different way, although this will not lead to any sharper ranges in the applications. Hence, we relegate this proof to the Appendix. If the reader is only interested in the applications to coinvariant algebras and Garsia-Haiman modules, we recommend the reader read the Appendix in place of Section \ref{transferSec}. 

\subsection*{Acknowledgments}
I would like to thank Jennifer Wang, Nir Gadish, George Seelinger, and my advisor Jim Haglund for helpful discussions and insightful comments.

\section{Preliminaries and definitions}\label{prelims}
We will denote the irreducible representation of $S_n$, called a Specht module, indexed by partition $\lambda\vdash n$ by $\mathbb{S}^{\lambda}$. If $\lambda$ is a partition of some $m$ and $n-|\lambda|\geq \lambda_1$, then let $\lambda[n]:=(n-|\lambda|,\lambda_1,\hdots,\lambda_{l(\lambda)})$. Note that every partition can be put into the form $\lambda[n]$ for a unique $\lambda$ and $n$. We will sometimes use $\lambda$ to refer to the set of cells in its Young diagram and let $H\subseteq_{\text{diag}} \lambda$ be subset of these cells, while $\mu\subseteq_{\text{set}}\lambda$ will denote a subset of the multi-set of parts of $\lambda$, and $\mu\subseteq \lambda$ will denote a sub-partition. As usual, $l(\lambda)$ is the length of a partition and $\lambda'$ is the conjugate partition to $\lambda$. Partions will be drawn in French notation.\parspace

We summarize the definition of representation stability introduced by Church, Ellenberg, and Farb in \cite{FI-module}. Let $\textbf{FI}$ be the category of finite sets with injections as morphisms. An \textbf{FI}-module is a functor, $F$, from \textbf{FI} to the category of $\C$-modules, $\textbf{Mod}_\C$. Throughout the paper, we will work with the field $\C$, although it may be replaced by any characteristic 0 field. Letting $[n]=\{1,\hdots,n\}$ and $\iota_n:[n]\rightarrow [n+1]$ the injection mapping $i$ to $i$, we can set $V_n=F([n])$ and $\phi_n=F(\iota_n)$. Then $F$ induces an $S_n$-module structure on $V_n$ by considering bijections on $[n]$ and furthermore we get a sequence 
\begin{equation}\label{FI-seq}
    V_0\overset{\phi_0}{\longrightarrow}V_1\overset{\phi_1}{\longrightarrow}V_2\overset{\phi_2}{\longrightarrow}\cdots
\end{equation}
of $S_n$-modules for varying $n$. By considering the standard embedding $S_n \hookrightarrow S_{n+1}$ by letting $\sigma\in S_n$ fix $n+1$, the connective maps $\phi_n:V_n\rightarrow V_{n+1}|_{S_n}$ are $S_n$-equivariant. It is not hard to verify that an equivalent way of thinking about \textbf{FI}-modules is taking a sequence $(V_n\in \textbf{Mod}_{\C[S_n]})_{n\in \N}$ with connective maps $\phi_n$ as in (\ref{FI-seq}) that are $S_n$-equivariant and also satisfy 
\begin{equation}\label{FIcriterion}
    (n+1~n+2)\cdot\phi_{n+1}\circ \phi_n=\phi_{n+1}\circ\phi_n,
\end{equation}
a condition that is forced by functoriality. For this reason, we will think about \textbf{FI}-modules and sequences (\ref{FI-seq}) with $S_n$-equivariant $\phi_n$ satisfying (\ref{FIcriterion}), interchangeably.\parspace

\begin{defi}
    Consider a sequence $(V_n\in \textbf{Mod}_{\C[S_n]})_{n\in \N}$ with decompositions
    $$V_n\cong \bigoplus_{\lambda}(\mathbb{S}^{\lambda[n]})^{\oplus c_{\lambda,n}},$$
    where $\lambda$ ranges over all partitions and $c_{\lambda,n}=0$, when $\lambda[n]$ is not defined. We say that $V_\bullet$ is \textbf{representation multiplicity stable (RMS)}, if for all partitions $\lambda$, there is an $N_\lambda$, such that for all $n\geq N_\lambda$, the coefficients $c_{\lambda,n}=c_{\lambda,n+1}$ have stabilized.\parspace

    We say that $V_\bullet$ is \textbf{uniformally representation multiplicity stable (URMS)}, if there is an $N$, such that for all $n\geq N$ and all $\lambda$, the coefficients $c_{\lambda,n}=c_{\lambda,n+1}$ have stabilized. This $N$ is called a \textbf{URMS stable range}.\parspace

    Following \cite{FI-module}, a sequence $(V_n\in \textbf{Mod}_{\C[S_n]})_{n\in \N}$ with $S_n$-equivariant connective maps $\phi_n$ satisfying (\ref{FIcriterion}) is said to be \textbf{uniformally representation stable (URS)}, if there is a URMS stable range, $N$, such that the following two additional conditions hold for $n\geq N$: the map $\phi_n$ is injective and $S_{n+1}\cdot \phi_n(V_n)=V_{n+1}$. We will call such an $N$ a \textbf{URS stable range}.
\end{defi}

Notice that the URMS condition does not depend on any conditions on connective maps between the $V_n$, while the URS condition does. See \cite{bahran} for a summary of other types of stable ranges that have appeared in the literature.\parspace

The authors of \cite{FI-module} showed that an \textbf{FI}-module is finitely generated if and only if the corresponding sequence $(V_n,\phi_n)$ is uniformaly representation stable (Proposition 3.3.3 \cite{FI-module}). In addition, they show that the category of \textbf{FI}-modules is noetherian and hence, any \textbf{FI}-submodule of a finitely generated \textbf{FI}-module is finitely generated (Theorem 1.3 \cite{FI-module}). We will use this fact to see the representation stability of sequences of Garsia-Haiman modules in Section \ref{macSec}.\parspace

One important quantity related to \textbf{FI}-modules is the \textbf{weight}, which may be defined for any sequence of $S_n$-modules $V_\bullet =(V_n\in\text{Mod}_{S_n})_{n\in \N}$ as follows
$$\wt(V_\bullet):=\sup_{\lambda,~n}\left\{|\lambda|:~[\S^{\lambda[n]},V_n]\neq 0\right\},$$
where the supremum is taken over partitions $\lambda$ and $n\geq |\lambda|+\lambda_1$ and we recall that $[V,W]=\dim_\C\text{Hom}_{\C[S_n]}(V,W)$ for $S_n$-modules $V,W$. Note that by definition, this quantity is finite for URMS sequences. The weight quantity will be crucial to the statement of several results in the paper and so we are interested in computing or at least bounding the weight. See the end of Section \ref{torFreeSec} for a discussion of one good approach.\parspace

\subsection{Torsion-free \textbf{FI}-modules}\label{torFreeSec}
We say an \textbf{FI}-module is \textbf{torsion-free}, if the image of any morphism is an injective linear map. This is just equivalent to the the $\phi_n$ in the corresponding sequence $(V_n,\phi_n)_{n\in \N}$ being injective, since the functor necessarily maps isomorphisms in \textbf{FI} to bijective linear maps.\parspace

Recall the definition of the central primitive idempotents: for $\lambda \vdash n$
$$\text{X}^\lambda=\frac{f_\lambda}{n!}\sum_{\sigma\in S_n}\chi^\lambda(\sigma)\sigma\in Z(\C[S_{n}]),$$
which satisfy $(\text{X}^\lambda)^2=\text{X}^\lambda$ and $\text{X}^\lambda\text{X}^\mu=0$ for $\lambda,\mu$ distinct partitions of $n$. Note that we can naturally view $\text{X}^\lambda$ as an element $\C[S_m]$ for $m\geq n$ and hence apply it to elements of $S_m$-representations. In particular, if $V$ is an $S_m$-representation and $\lambda\vdash n\leq m$, we have that $\text{X}^\lambda$ is an $S_n$-equivariant projection from $V$ onto an $S_n$-submodule of $V|_{S_n}=\text{Res}^{S_m}_{S_n}(V)$ that is $S_n$-isomorphic to $(\S^\lambda)^{\oplus c}$, where $c=[V|_{S_n},\S^\lambda]$. We will use these operators in the proof of the following lemma.\parspace

\begin{lem}\label{monotonictyLemma}(Monotonicity lemma)
    Let $(V_n,\phi_n)_n$ be a sequence arising from a torsion-free \textbf{FI}-module, i.e. each $\phi_n$ is injective. Then for all partitions $\lambda$ and $n\geq |\lambda|+\lambda_1$,
    \begin{equation}\label{monotonicity}
        \left[V_n,\mathbb{S}^{\lambda[n]}\right]\leq \left[V_{n+1},\mathbb{S}^{\lambda[n+1]}\right].
    \end{equation}
\end{lem}

\begin{proof}
    Recall the definition of the natural torsion-free \textbf{FI}-module, $(\mathbb{S}^{\lambda[\bullet ]},\psi_\bullet)$, called $V(\lambda)$ in \cite{FI-module}. Here $\mathbb{S}^{\lambda[n]}$ is taken to be the zero module for $n<|\lambda|+\lambda_1$ and for $n\geq |\lambda|+\lambda_1$, $\psi_n:\mathbb{S}^{\lambda[n]}\rightarrow \mathbb{S}^{\lambda[n+1]}$ is the unique (up to scaling) injective $S_n$-equivariant map between the two modules.\parspace

    Then we get a torsion-free \textbf{FI}-module
    $$\dots \overset{\phi_{n-1}\otimes \psi_{n-1}}{\longrightarrow}V_{n}\otimes \mathbb{S}^{\lambda[n]}\overset{\phi_n\otimes \psi_n}{\longrightarrow}V_{n+1}\otimes \mathbb{S}^{\lambda[n+1]}\overset{\phi_{n+1}\otimes \psi_{n+1}}{\longrightarrow}\dots$$
    and the self-duality of irreducible $S_n$-representations implies 
    $$[V_n\otimes \mathbb{S}^{\lambda[n]},\mathbb{S}^{(n)}]=\dim\text{Hom}_{S_n}(V_n\otimes \mathbb{S}^{\lambda[n]},\S^{(n)})=\dim\text{Hom}_{S_n}(V_n,\S^{\lambda[n]}\otimes \S^{(n)})=[V_n,\mathbb{S}^{\lambda[n]}].$$
    Hence it suffices to prove the lemma for the case $\lambda=\emptyset$, which can be restated in terms of invariants: given a torsion-free \textbf{FI}-module $(V_n,\phi_n)_n$,
    $$\dim((V_n)^{S_n})\leq \dim((V_{n+1})^{S_{n+1}}).$$
    Suppose for contradiction the statement failed for some $(V_n,\phi_n)_n$ and there was a basis of $S_n$-invariants of $V_n$, $\{v_1,v_2,\dots,v_c\}$, but $\dim((V_{n+1})^{S_{n+1}})<c$. Then consider the $S_{n+1}$-submodule of $V_{n+1}$ generated by $\phi_n(v_1),\cdots,\phi_n(v_c)$:
    $$W:=S_{n+1} \langle \phi_n(v_1),\cdots,\phi_n(v_c)\rangle.$$
    Since $\text{Span}_\C(\phi_n(v_1),\cdots,\phi_n(v_c))$ forms an $S_n$-module isomorphic to $(\S^{(n)})^{\oplus c}$, $W$ is $S_{n+1}$-isomorphic to some quotient of 
    $$S_{n+1}\otimes_{S_n}(\S^{(n)})^{\oplus c}=\text{Ind}_{S_n}^{S_{n+1}}\p{(\S^{(n)})^{\oplus c}}\cong_{S_{n+1}} (\S^{(n+1)})^{\oplus c}\oplus (\S^{(n,1)})^{\oplus c},$$
    by the branching rule. Hence 
    $$W\cong_{S_{n+1}} (\S^{(n+1)})^{\oplus a}\oplus (\S^{(n,1)})^{\oplus b},$$
    for some $a,b\leq c$. Note that $a=\dim(W^{S_{n+1}})\leq \dim((V_{n+1})^{S_{n+1}})<c$. Then applying the projection operator, $\text{X}^{(n+1)}$, to $\text{Span}_\C(\phi_n(v_1),\cdots,\phi_n(v_c))$ yields a subspace of an $a$-dimensional space and so there is a non-trivial linear relation:
    $$\alpha_1\cdot \text{X}^{(n+1)}\phi_n(v_1)+\cdots +\alpha_c\cdot \text{X}^{(n+1)}\phi_n(v_c)=0.$$
    Let $v=\alpha_1\cdot v_1+\cdots +\alpha_c\cdot v_c\neq 0$, since the $v_i$ form a basis. Since $v$ is an $S_n$-invariant, $S_{n+1}\langle \phi_n(v)\rangle$ is isomorphic to a quotient of $\S^{(n+1)}\oplus \S^{(n,1)}$. But since $\text{X}^{(n+1)}\phi_n(v)=0$ and $\phi_n$ is injective, we conclude that 
    $$S_{n+1}\langle \phi_n(v)\rangle\cong_{S_{n+1}} \mathbb{S}^{(n,1)}.$$
    While an exceptional case can occur for small $n$, we will demonstrate the contradiction when $n\geq 3$ and return to the case $n<3$ later.\parspace

    By the functoriality of \textbf{FI}-modules, one has 
    $$(n+1~n+2)\cdot \phi_{n+1}\circ \phi_n(v)=\phi_{n+1}\circ \phi_n(v)$$
    and so $\phi_{n+1}\circ \phi_n(v)$ is a non-trivial $(S_n\times S_2)$- invariant in $V_{n+2}$. By the Pieri rule, the only irreducible representations of $S_{n+2}$ with a non-trivial $(S_n\times S_2)$-invariant are $\mathbb{S}^{(n+2)},\mathbb{S}^{(n+1,1)},\mathbb{S}^{(n,2)}$. Furthermore, this $(S_n\times S_2)$-invariant is unique up to scaling. Hence, $\phi_{n+1}$ induces an injective $S_{n+1}$-equivariant map from $S_{n+1}\langle \phi_n(v)\rangle\cong_{S_{n+1}} \mathbb{S}^{(n,1)}$ to $\mathbb{S}^\mu$:
    $$\rho: S_{n+1}\langle \phi_n(v)\rangle\rightarrow \S^\mu|_{S_{n+1}},$$
    for some $\mu=(n+2),(n+1,1),(n,2)$ such that it maps $\phi_n(v)$ to the unique (up to scaling) $(S_n\times S_2)$-invariant of $\S^\mu$, call it $f$. Note that $\mu\neq (n+2)$ since there is no injective $S_{n+1}$-equivariant map from $\S^{(n,1)}\rightarrow \S^{(n+2)}$ by the branching rule.\parspace

    In the case $\mu=(n,2)$, since we are taking $n\geq 3$, the partition, $(n-1,2)$, is well-defined and $\text{X}^{(n-1,2)}\phi_n(v)=0$, since $S_{n+1}\langle \phi_n(v)\rangle\cong \S^{(n,1)}$ doesn't contain $\S^{(n-1,2)}$ as a submodule. We can argue that 
    $$0\neq \text{X}^{(n-1,2)}f=\text{X}^{(n-1,2)}\rho(\phi_n(v))=\rho(\text{X}^{(n-1,2)}\phi_n(v)),$$
    creating a contradiction. Decompose the restriction 
    $$\S^{(n,2)}|_{S_{n+1}}= U\oplus W,$$
    where $U$ is isomorphic to $\S^{(n,1)}$ and $W$ is isomorphic to $\S^{(n-1,2)}$. Then we can write $f=g+h$ for $g\in U$ and $h\in W$. Suppose that $h=0$ and $f=g\in U$. Then for any $\sigma\in S_{n+2}$, we may write $\sigma=\pi\circ(n+1~n+2)$ or $\sigma=\pi$ for some $\pi \in S_{n+1}$. Since $g=f$ is an $(S_n\times S_2)$-invariant (so that $(n+1~n+2)\cdot g=g$) and $S_{n+1}\cdot U\subset U$, we get $\sigma \cdot g\in U$ in either case. But then $S_{n+2}\langle g\rangle=U\subsetneq \S^{(n,2)}$, contradicting irreducibility of $\S^{(n,2)}$. Hence $0\neq h=\text{X}^{(n-1,2)}f$.\parspace
    
    In the case $\mu=(n+1,1)$, we can similarly argue that $\text{X}^{(n+1)}\phi_n(v)=0$, but $\text{X}^{(n+1)}f\neq 0$, yielding another contradiction. To see this consider the restriction 
    $$\S^{(n+1,1)}|_{S_{n+1}}\cong U\oplus W,$$
    where $U$ is isomorphic to $\S^{(n,1)}$ and $W$ is isomorphic to $\S^{(n+1)}$. Then we can write $f=g+h$ for $g\in U$ and $h\in W$. Again, we can argue that if $h=0$, then $S_{n+2}\langle g\rangle=U\subsetneq \S^{(n+1,1)}$. Hence $h=\text{X}^{(n+1)}f\neq 0$. Hence all three possible cases for $\mu$ yield a contradiction, forcing our initial assumption that $\dim((V_{n+1})^{S_{n+1}})<\dim((V_n)^{S_n})$ to be false for $n\geq 3$.\parspace

    Finally, we cover the cases $n<3$. The case $n=0$ is trivial. In the case $n=1$, we are forced to have a $S_{2}$-equivariant map from $\S^{(1,1)}$ to $\S^{(2,1)}$ that maps the generator of $\S^{(1,1)}$ to the $(S_1\times S_2)$-invariant of $\S^{(2,1)}$, but this is not an $S_2$-equivariant map by the $\mu=(n+1,1)$ case above. For $n=2$, we again can rule out $\mu=(3,1)$, but it turns out that there is in fact a $S_{3}$-equivariant map from $\S^{(2,1)}$ to $\S^{(2,2)}$ mapping the $(S_2\times S_1)$-invariant of $\S^{(2,1)}$ to the $(S_2\times S_2)$-invariant of $\S^{(2,2)}$. In this case, we just consider the argument for the next connective map, $\phi_{4}$, which induces an $S_{4}$-equivariant map from $\S^{(2,2)}$ to $\S^{\mu}|_{S_4}$ for $\mu\vdash 5$, sending the $(S_2\times S_2)$- invariant, $f$, to an $(S_2\times S_3)$-invariant, $g$, unique up to scalar multiple, by functoriality of the \textbf{FI}-module. This combined with the branching rule forces $\mu=(3,2)$, but this would not yield a $S_4$-equivariant map, since one can easily check that $\text{X}^{(3)}f=0$, but $\text{X}^{(3)}g\neq 0$ by the branching rule. 
\end{proof}

\begin{rem}
    We interpret the proof above as proving the following representation theoretic fact for $n\neq 2$. Suppose $V_{n+1}$ is an $S_{n+1}$-module and $V_n$ is an $S_n$-invariant subspace. Then to show that $[V_n,\S^{\lambda[n]}]\leq[V_{n+1},\S^{\lambda[n+1]}]$, it suffices to find an $S_{n+2}$-module $V_{n+2}$ containing $V_{n+1}$ (with compatible $S_{n+1}$ action on $V_{n+1}$), such that $(n+1~n+2)$ acts as the identity on $V_n$. 
\end{rem}

This yields the following easy corollary, which establishes that for torsion-free \textbf{FI}-modules it suffices to show that Specht module multiplicities have stabilized in order to establish representation stability in the sense of Church, Ellenberg, and Farb.

\begin{cor}\label{URMS->URS}
    Let $(V_n,\phi_n)_n$ be a sequence arising from a finitely generated torsion-free \textbf{FI}-module. Then a URMS stable range is also a URS stable range.
\end{cor}

\begin{proof}
    If $N$ is a URMS stable range, then we have stable decompositions for $n\geq N$
    $$V_n\cong_{S_n}\bigoplus_{\lambda}(\S^{\lambda[n]})^{\oplus c_\lambda}$$
    for $c_\lambda\in \N$ independent of $n$. Since the maps $\phi_n$ are automatically injective, we only have to verify that $W:=S_{n+1}\langle \phi_n(V_n)\rangle=V_{n+1}$ for $n\geq N$. We can look at the modified torsion-free \textbf{FI}-module with corresponding sequence
    $$\cdots \overset{\phi_{n-1}}{\longrightarrow}V_n\overset{\phi_{n}}{\longrightarrow}W\overset{\phi_{n+1}|_{W}}{\longrightarrow}V_{n+2}\overset{\phi_{n+2}}{\longrightarrow}\cdots$$
    Since $W\subseteq V_{n+1}$, we have that $[W,\S^{\lambda[n+1]}]\leq c_\lambda$ for all relevant $\lambda$. On the other hand, applying the Lemma \ref{monotonictyLemma} to the sequence above implies $[W,\S^{\lambda[n+1]}]\geq c_\lambda$. This forces us to conclude that $W\cong \bigoplus_{\lambda}(\S^{\lambda[n+1]})^{\oplus c_\lambda}\cong V_{n+1}$ and so $W=V_{n+1}$ as claimed.
\end{proof}

Another nice property of torsion-free \textbf{FI}-modules, with corresponding sequence $(V_n,\phi_n)_n$, is that the function $n\mapsto \dim(V_n)$ is closely related to the weight of the \textbf{FI}-module. It is well-known that $\dim\p{\S^{\lambda[n]}}$ is agree with a polynomial in $n$ of degree $|\lambda|$ for sufficiently large $n$ (see \cite{polyDim} for explicit combinatorial formulas for this polynomial). In general, knowing the function $n\mapsto \dim(V_n)$ does not necessarily give us information about the weight, nor does it allow us to conclude representation stability. For instance, the non-torsion-free \textbf{FI}-module corresponding to the sign representation
$$V_n=\S^{(1^n)},~\phi_n=0,$$
has $\wt(V_\bullet)=\infty$ but $\dim(V_n)=1$, a degree 0 polynomial. We bring attention to the fact that we cannot make take consecutive $\phi_n$ to be non-zero while respecting (\ref{FIcriterion}). As another example, one can have a torsion-free \textbf{FI}-module
$$V_n=(\S^{(n)})^{\oplus p(n)},$$
with $\phi_n$ taken to be arbitrary injective linear maps and $p(n)$ is any non-decreasing polynomial. The $\wt(V_\bullet)=0$, but $p(n)$ can have any degree.\parspace

However, for \textit{finitely generated}, torsion-free \textbf{FI}-modules, it follows by Lemma \ref{monotonictyLemma} that $\dim(V_n)$ is eventually a polynomial in $n$ of degree $\wt(V_\bullet)$. Even without establishing that the sequence $(V_n,\phi_n)_n$ is finitely generated (URS), we can still bound the weight by using its dimension as seen in the following lemma.

\begin{lem}\label{weightDegree}
    Let $(V_n,\phi_n)_n$ be a sequence arising from a torsion-free \textbf{FI}-module, such that $\dim(V_n)=O(n^k)$, for some $k\in\N$. Then the weight of the \textbf{FI}-module is $\leq k$.
\end{lem}

\begin{proof}
    The hypothesis implies that $\limsup_{n\rightarrow \infty}\frac{\dim(V_n)}{n^k}<\infty$. Now suppose for contradiction that the weight of the \textbf{FI}-module is not $\leq k$, so there is some $\lambda\vdash K>k$ and $n_0\geq |\lambda|+\lambda_1$, such that a copy of $\S^{\lambda[n_0]}$ appears in the decomposition of $V_{n_0}$. By Lemma \ref{monotonictyLemma}, a copy of $\S^{\lambda[n]}$ will appear in $V_n$ for all $n> n_0$ and so 
    $$\dim(V_n)=\dim(\S^{\lambda[n]})+f(n),$$
    where $f(n)$ is some positive function of $n$ corresponding to the dimension of the quotient space. But since $\dim(\S^{\lambda[n]})$ is a eventually polynomial in $n$ in degree $>k$, $\limsup_{n\rightarrow \infty}\frac{\dim(V_n)}{n^k}\geq \limsup_{n\rightarrow \infty}\frac{\dim(\S^{\lambda[n]})}{n^k}=\infty$, a contradiction.
\end{proof}

One might ask if for a finitely generated, torsion-free \textbf{FI}-module the sharpest stable range $n\geq N$, for which $\dim(V_n)$ starts to agree with a polynomial in $n$, agrees with the sharpest URS stable range. We note that the \textbf{FI}-module, $(\S^{(1^k)[\bullet]},\psi_\bullet)$, appearing in the proof of the monotonicity lemma has $\dim(\S^{(1^k)[n]})=\binom{n-1}{k}$ stably agreeing with a polynomial for $n\geq 1$. However, the sharpest URS stable range is $n\geq k+1$, since otherwise $(1^k)[n]$ is undefined. Understanding how these two stable ranges relate, may be an interesting further problem.

\subsection{Symmetric function theory}

We refer the reader to \cite{stanley2} and \cite{macdonald1998symmetric} for an introduction to the theory of symmetric functions. Let $\Lambda$ denote the set of symmetric functions over $\C$ in variables $x_1,x_2,\hdots$ and let $\Lambda^{(n)}$ denote the $n$-th degree homogeneous component.\parspace

Recall some of the classical homogeneous bases for $\Lambda$, which are indexed by partitions with $u_\lambda\in \Lambda^{(n)}$ whenever $\lambda\vdash n$. Let $(m_\lambda)_{\lambda}$ denote the monomial symmetric function basis, $(h_\lambda)_{\lambda}$ denote the complete homogeneous basis, $(e_\lambda)_{\lambda}$ denote the elementary basis, $(p_\lambda)_{\lambda}$ denote the power sum basis, and $(s_\lambda)_{\lambda}$ denote the Schur function basis.\parspace

Define the Grothendieck ring of $(S_n)_{n\in \N^+}$, $R=\oplus_n R^{(n)}$, where $R^{(n)}$ is the free $\C$-module generated by isomorphism classes of irreducible representations of $S_n$. Letting $[V]$ denote the isomorphism class of an $S_n$-representation, we can identify $[V\oplus W]$ with $[V]+[W]$ in $R^{(n)}$, thereby allowing us to interpret the isomorphism class of \textit{any} $S_n$-representation as an element of $R^{(n)}$ by Maschke's theorem. The product in the ring is given by the induced tenor product: if $V$ is an $S_n$-representation and $W$ is an $S_m$-representation, then $[V]\cdot [W]:=\text{Ind}_{S_n\times S_m}^{S_{n+m}}(V\otimes W)$, interpreting $S_n\times S_m$ as a Young subgroup of $S_{n+m}$.\parspace

Then one can define the Frobenius characteristic map, $\mathcal{F}:R^{(n)}\rightarrow \Lambda^{(n)}$, to be the map that sends $\mathcal{F}(\mathbb{S}^{\lambda})$ to $s_\lambda$ and is extended by linearity. Then by taking the direct sum over graded components yields a map $\mathcal{F}:R\rightarrow \Lambda$ which is a graded $\C$-algebra isomorphism of the Grothendieck ring and the algebra of symmetric functions.\parspace

Taking the natural inner product on characters of $S_n$-representations, $\mathcal{F}$, induces an inner product on $\Lambda^{(n)}$ called the Hall scalar inner product, which we will just denote by $\langle \cdot,\cdot\rangle$. A pair of homogeneous basis $(u_\lambda\in \Lambda^{(n)})_{\lambda\vdash n}$, $(v_\lambda\in \Lambda^{(n)})_{\lambda\vdash n}$, are dual if $\langle u_\lambda,v_\mu\rangle=\delta_{\lambda,\mu}$. It is well-known that $(h_\lambda)_\lambda$, $(m_\mu)_\mu$ form a dual basis pair, as do $(p_\lambda)_\lambda$, $(p_\mu/z_\mu)_\mu$, and the Schur functions are self-dual, being the images of the irreducible representations of $S_n$. Here, 
$$z_\mu:=\mu_1\mu_2\cdots \mu_{l(\mu)}\cdot \prod_{j\geq 1}c_j(\mu)!$$
for $\mu=(\mu_1,\cdots,\mu_{l(\mu)})$, where $c_j(\mu)$ is the number of parts of $\mu$ of size $j$. This is the order of the centralizer of a $\sigma\in S_n$ of cycle type $\mu$.\parspace

We have the following important identity for dual basis pairs. If $(u_\lambda),(v_\mu)$ are a dual pair and $(f_\lambda),(g_\mu)$ are a dual pair and we have a change of basis
$$u_\lambda=\sum_{\mu\vdash n}A_{\lambda,\mu}f_\mu$$
for $A_{\lambda,\mu}\in\C$, then the dual (or adjoint) identity is
$$g_\mu=\sum_{\lambda\vdash n}A_{\lambda,\mu}v_\lambda.$$

We now go through the various change of basis transformations between these bases. See \cite{egeciogluBrickTab} for a more comprehensive overview of the change of basis coefficients. We start with some of the following change of bases
\begin{align*}
    h_{\lambda}&=\sum_{\mu\vdash n}IM_{\lambda,\mu}m_\mu\\
    e_\lambda&=\sum_{\mu\vdash n}BM_{\lambda,\mu}m_\mu\\
    p_\lambda&=\sum_{\mu\vdash n}|OB_{\lambda,\mu}|\, m_\mu\\
    h_\lambda&=\sum_{\mu\vdash n}(-1)^{n-l(\mu)}|B_{\lambda,\mu}|\, e_\mu\\
    p_\lambda&=\sum_{\mu\vdash n}(-1)^{l(\lambda)-l(\mu)}w(B_{\mu,\lambda}) h_\mu\\
    s_\lambda&=\sum_{\mu\vdash n}K_{\lambda,\mu}m_\mu.
\end{align*}
Here $IM_{\lambda,\mu}$ counts the number of $l(\lambda)\times l(\mu)$ matrices with natural number entries, $A\in M_{l(\lambda),l(\mu)}(\N)$, such that the sum of the entries in the $i$-th row is $\lambda_i$ and the sum of the entries in the $j$-th column is $\mu_j$. Similarly, $BM_{\lambda,\mu}$ is the number of such matrices only using the entries $0,1$. We will refer to the coefficients of the inverse change of basis transformation $IM^{-1}_{\lambda,\mu}$ (resp. $BM^{-1}_{\lambda,\mu}$), which have combinatorial formulas in terms of primitive bi-brick tableaux which can be found in \cite{primBrick}.\parspace

The so-called ordered brick tableau, $OB_{\lambda,\mu}$ are maps $f$ from the distinguished parts of $\lambda$ to the distinguished parts of $\mu$ such that the parts in the pre-image of each $\mu_i$ sum to $\mu_i$. The set of brick tableau, $B_{\lambda,\mu}$ consists of laying ``bricks'' specified  by the parts of $\mu$, where equal sized parts of $\mu$ are not distinguished, into the rows of $\lambda$. The weight of a brick tableau, $T$, given by $w(T)$ is the product of the rightmost parts of $\mu$ in each row of $\lambda$. The total weight, $w(B_{\mu,\lambda})=\sum_{T\in B_{\mu,\lambda}}w(T)$. See Fig \ref{brickTab}.\parspace
 
\begin{figure}
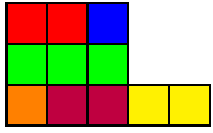

    \centering
    \begin{ytableau}
*(red) & *(red)  &*(blue)  \\
*(green)  & *(green)  &*(green)\\
*(orange)& *(purple)& *(purple)& *(yellow)& *(yellow)
\end{ytableau}
    \caption{A brick tableau of shape $\lambda=(5,3^2)$ and content $\mu=(3,2^3,1^2)$. The weight is $1\cdot 3\cdot 2=6$.}
    \label{brickTab}
\end{figure}

The Kostka numbers $K_{\lambda,\mu}:=\#\text{SSYT}(\lambda,\mu)$ count the number of semi-standard Young tableau of shape $\lambda$ and content $\mu$. Since these are natural numbers, it follows that the Frobenius image of any $S_n$-module, will be a non-negative sum of Schur functions and hence also a non-negative sum of monomial symmetric functions. Furthermore, the monotonicity observed in Lemma \ref{monotonictyLemma} is passed down to the coefficients of the monomial expansion as well.

\begin{cor}\label{monomialMonotonicity}
    Let $(V_n,\phi_n)_n$ be a sequence arising from a torsion-free \textbf{FI}-module. Then the Schur coefficients of the Frobenius images satisfy
    \begin{equation}\label{schurIneq}
        \langle \mathcal{F}(V_n),s_{\lambda[n]\rangle}\rangle\leq \langle \mathcal{F}(V_{n+1}),s_{\lambda[n+1]\rangle}\rangle,
    \end{equation}
    for partition $\lambda$ and $n\geq |\lambda|+\lambda_1$. The monomial coefficients satisfy
    \begin{equation}\label{monomIneq}
        \langle \mathcal{F}(V_n),h_{\mu[n]\rangle}\rangle\leq \langle \mathcal{F}(V_{n+1}),h_{\mu[n+1]\rangle}\rangle,
    \end{equation}
    for partition $\mu$ and $n\geq |\mu|+\mu_1$.
\end{cor}

\begin{proof}
    The first inequality is an immediate consequence of Lemma \ref{monotonictyLemma}. Since $\langle s_{\lambda[n]},h_{\mu[n]}\rangle=K_{\lambda[n],\mu[n]}$ by the duality of $(h_\mu)$ and $(m_\mu)$, we have that
    \begin{align*}
       \langle \mathcal{F}(V_n),h_{\mu[n]}\rangle&=\sum_{\lambda}K_{\lambda[n],\mu[n]}\cdot \langle\mathcal{F}(V_n),s_{\lambda[n]}\rangle\\
       &\leq \sum_{\lambda}K_{\lambda[n+1],\mu[n+1]}\cdot \langle\mathcal{F}(V_{n+1}),s_{\lambda[n+1]}\rangle\\
       &=\langle \mathcal{F}(V_{n+1}),h_{\mu[n+1]}\rangle
    \end{align*}
    using the first inequality in the statement and $K_{\lambda[n],\mu[n]}\leq K_{\lambda[n+1],\mu[n+1]}$ which follows by the injective map constructed in the proof of Lemma \ref{kostkaStable}.
\end{proof}

For the inverse change of basis we have the inverse Kostka numbers
$$m_\mu=\sum_{\lambda\vdash n}K^{-1}_{\mu,\lambda}s_\lambda,$$
a formula for which is given as follows. A rim hook of $\lambda$ is a set $H\subseteq_{\text{diag}} \lambda$ of cells encountered by taking a walk in $\lambda$ consisting of east and south steps, such that the removal of these cells still leaves a partition shape. A rim hook is special if it contains a cell from the first row of $\lambda$. A special rim hook tableau of shape $\lambda$ with content $\mu$, is obtained by successively removing special rim hooks $H_1,H_2,\hdots,H_{l(\mu)}$ from $\lambda$ until one gets the empty partition such that the $|H_1|,|H_2|,\hdots,|H_{l(\mu)}|$ are a rearrangement of the parts of $\mu$. The sign of a special rim hook tableau is $\sgn(T)=\prod_{i}(-1)^{\#\text{ south steps in }H_i}$. Let $\text{SRHT}(\lambda,\mu)$ denote the set of special rim hook tableau of shape $\lambda$ and content $\mu$. See Fig \ref{srht}. Then
$$K^{-1}_{\mu,\lambda}=\sum_{T\in \text{SRHT}(\lambda,\mu)}\sgn(T).$$

\begin{figure}
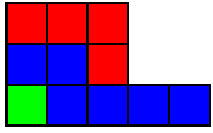

    \centering
    \begin{ytableau}
*(red) & *(red)  &*(red)  \\
*(blue)  & *(blue)  &*(red)\\
*(green)& *(blue)& *(blue)& *(blue)& *(blue)
\end{ytableau}
    \caption{An example of special rim hook tableau of shape $\lambda=(5,3,3)$ and content $\mu=(6,4,1)$, with the special rim hooks highlighted in red,blue, and green. The sign is $(-1)^2=1$.}
    \label{srht}
\end{figure}

Recall the character change of basis
$$s_\lambda=\sum_{\mu\vdash n}\chi^{\lambda}_\mu \frac{p_\mu}{z_\mu},$$
where $\chi^{\lambda}_\mu$ is the evaluation of the character of $\mathbb{S}^{\lambda}$ at a permutation of cycle type $\mu$ (this follows from the fact that the Frobenius map sends the indicator class function of the conjugacy class of cycle type $\mu$ to $p_\mu/z_\mu$). These coefficients are given by the Murnaghan-Nakayama rule 
$$\chi^\lambda_{\mu}=\sum_{T\in \text{RHT}(\lambda,\alpha)}\sgn(T),$$
where $\alpha$ is any \textit{fixed} rearrangement of $\mu$ and $\text{RHT}(\lambda,\alpha)$ is the set of rim hook tableau of shape $\lambda$ and content $\alpha$. This is the set of rim hook tableau where we do not require the hooks to be special (they can start in arbitrary columns of $\lambda$), but the hooks must be removed in the order specified by $\alpha$: $(|H_1|,\hdots,|H_{l(\mu})|)=(\alpha_1,\hdots,\alpha_{l(\mu)})$. See Table \ref{basisTable} for the remaining change of basis coefficients for basis pairs that we did not mention here.

\section{Transferring stability between symmetric function bases}\label{transferSec}

The aim of this section is to provide a careful analysis of how one can use the change of basis coefficients between two symmetric function basis to
deduce how stability properties of expansions in one symmetric function basis imply stability properties of expansions in another basis. Since we consider stabilizing expansions of homogeneous $F_n\in \Lambda^{(n)}$ as $n$ increases, it makes sense to restrict our attention to homogeneous basis for symmetric functions indexed by partitions. First, we generalize the notion of stability with respect to a basis.

\begin{defi}
     Let $(f_\lambda)_{\lambda}$ be a basis for $\Lambda$ with $f_\lambda\in \Lambda^{(|\lambda|)}$. Consider a sequence $(F_n)_n$ with $F_n\in \Lambda^{(n)}$ denoting the expansion in $f$ by $F_n=\sum_{\lambda}d_{\lambda,n}f_\lambda$. Define
     the $f$-\textbf{range} of $F_\bullet$, $\rg^f_\lambda(F_\bullet)$ to be the smallest $n$ such that $d_{\lambda,n}=d_{\lambda,n+k}$ for all $k\geq 0$ and let $\rg^f_\lambda(F_\bullet)=\infty$ when no such $n$ exists. We say $F_\bullet$ is $f$-\textbf{stable} if $\rg^f_\lambda(F_\bullet)<\infty$ for all $\lambda$.\parspace

     Similarly, define the \textbf{uniform} $f$-\textbf{range} of $F_\bullet$, $\rg^f(F_\bullet)=\max_{\lambda}\p{\rg^f_\lambda(F_\bullet)}$. We say $F_\bullet$ is \textbf{uniformally} $f$-\textbf{stable} if $\rg^f(F_\bullet)<\infty$.\parspace

     Define the \textbf{$f$-weight} of $F_\bullet$ 
     $$\text{wt}_f(F_\bullet):=\max\{|\lambda|:~\exists\, n,~d_{\lambda,n}\neq 0\}.$$
\end{defi}

It is easy to see that $F_\bullet$ is uniformally $f$-stable if and only if it is $f$-stable and $\wt_f(F_\bullet)<\infty$.

\begin{defi}
    Given a pair of symmetric function basis $(g_\mu)_{\mu}$, $(f_\mu)_{\mu}$. We say that \textbf{(uniform) stability transfers} from $f$ to $g$, if whenever a sequence $(F_n)_n$ is (uniformally) $f$-stable, then it is also (uniformally) $g$-stable. We say there is \textbf{weak stability transfer} from $f$ to $g$, if whenever $(F_n)_n$ is uniformally $f$-stable, then it is $g$-stable. We say \textbf{finite weight transfers} from $f$ to $g$, if whenever $(F_n)_n$ has $\text{wt}_f(F_\bullet)<\infty$, then $\text{wt}_g(F_\bullet)<\infty$.
\end{defi}

Let $(g_\mu)_{\mu}$, $(f_\mu)_{\mu}$ be a pair of basis with change of basis coefficients $A_{\lambda,\mu}$, i.e.
$$f_\lambda=\sum_{\mu\vdash |\lambda|}A_{\lambda,\mu}g_\mu,$$
and fix the convention that $A_{\lambda[n],\mu[n]}=0$, whenever $\lambda[n]$ or $\mu[n]$ is not defined. Let $N^{f\rightarrow g}_{\lambda,\mu}$ be the smallest $n$ such that $A_{\lambda[n],\mu[n]}=A_{\lambda[n+k],\mu[n+k]}$ for all $k\geq 0$ and let $N^{f\rightarrow g}_{\lambda,\mu}=\infty$ if no such $n$ exists. Define
$$J^{f\rightarrow g}_{\mu}:=\{\lambda:\exists\, n,~A_{\lambda[n],\mu[n]}\neq 0\},$$
$$I^{f\rightarrow g}_{\lambda}:=\{\mu:\exists\, n,~A_{\lambda[n],\mu[n]}\neq 0\}.$$
The reason for writing $f\rightarrow g$ in the superscript and not $g\rightarrow f$ is that when studying stability transfer from $f$ to $g$, we actually have to look at expansions of $f$ in terms of $g$ as we will see now.

\begin{thm}\label{stabTransferResult}
    Let $(g_\mu)_{\mu}$, $(f_\lambda)_{\lambda}$ be a pair of homogeneous basis with change of basis coefficients $A_{\lambda,\mu}$. Consider the following three conditions:
\begin{enumerate}
    \item $N^{f\rightarrow g}_{\lambda,\mu}<\infty$, for all $\lambda,\mu$
    \item $|J^{f\rightarrow g}_\mu|<\infty$, for all $\mu$
    \item $|I^{f\rightarrow g}_\lambda|<\infty$, for all $\lambda$
\end{enumerate}

Then
\begin{itemize}
    \item There is \underline{finite weight transfer} from $f$ to $g$ if and only if condition \underline{3 holds}
    \item There is \underline{weak stability transfer} from $f$ to $g$ if and only if condition \underline{1 holds}
    \item There is \underline{stability transfer} from $f$ to $g$ if and only if conditions \underline{1,2 hold}
    \item There is \underline{uniform stability transfer} from $f$ to $g$ if and only if conditions \underline{1,3 hold}
\end{itemize}
\end{thm}

\begin{proof}
\textit{(finite weight transfer from $f$ to $g$)}: If condition 3 failed for some $\lambda$, then 
$$F_n:=\begin{cases}
    f_{\lambda[n]},&n\geq |\lambda|+\lambda_1\\
    0,&\text{else}
\end{cases}$$
would have finite $f$-weight but $F_n=\sum_{\mu}A_{\lambda[n],\mu[n]}g_{\mu[n]}$ would have infinite $g$-weight. On the other hand, if condition 3 holds, then any $F_n$ with finite $f$-weight, only supports $f_{\lambda[n]}$ for $|\lambda|<\wt_f(F_\bullet)$. By condition 3, $F_n$ can only support $g_{\mu[n]}$ for some $n$ if $\mu\in \bigcup_{|\lambda|\leq \wt_f(F_\bullet)}I^{f\rightarrow g}_\lambda$, a finite set. Hence 
\begin{equation}\label{3:ineq}
    \wt_g(F_\bullet)\leq \max\left\{|\mu|:~\mu\in I^{f\rightarrow g}_\lambda,~|\lambda|\leq \wt_f(F_\bullet)\right\}<\infty.
\end{equation}

\textit{(weak stability transfer from $f$ to $g$)}: The necessity of condition 1 follows using the uniformally $f$-stable sequence $F_n$ defined in the previous part.\parspace

Suppose we have a uniformally $f$-stable sequence 
\begin{align*}
F_n&=\sum_{|\lambda|\leq \wt_f(F_\bullet)}d_{\lambda,n}f_{\lambda[n]}\\
&=\sum_{\mu}\p{\sum_{|\lambda|\leq \wt_f(F_\bullet)}d_{\lambda,n}A_{\lambda[n],\mu[n]}}g_{\mu[n]}.
\end{align*}
Then the coefficient of $g_{\mu[n]}$ in $F_n$: 
$$\sum_{|\lambda|\leq \wt_f(F_\bullet)}d_{\lambda,n}A_{\lambda[n],\mu[n]}$$
stabilizes once $n\geq \rg^f_\lambda(F_\bullet)$ and $n\geq N^{f\rightarrow g}_{\lambda,\mu}$ for $|\lambda|\leq \wt_f(F_\bullet)$. So we get an upper bound on the corresponding range for the $g$ expansion
\begin{equation}\label{1:ineq}
    \rg^g_\mu(F_\bullet)\leq \max\p{\rg^f(F_\bullet),~\max\left\{N^{f\rightarrow g}_{\lambda,\mu}:~|\lambda|\leq \wt_f(F_\bullet)\right\}}<\infty
\end{equation}

\textit{(stability transfer from $f$ to $g$)}: If for some pair $\lambda,\mu$, the value $A_{\lambda[n],\mu[n]}$ did not stabilize, then the $F_n$ defined in the first part would be $f$-stable, but not $g$-stable since the coefficient of $g_{\mu[n]}$ in $F_n$ is $A_{\lambda[n],\mu[n]}$. This shows the necessity of condition 1.\parspace

Now suppose for contradiction that stability transfers from $f$ to $g$ but condition 2 is not true. Then for some $\mu$, $A_{\lambda[n],\mu[n]}\neq 0$ for infinitely many $\lambda$, $n$ pairs. Pick a sequence of such $\lambda$, $n$ values, calling them $\lambda^{(1)},\lambda^{(2)},\dots$ and $n^{(1)},n^{(2)},\dots$ Without loss of generality we can take the $n^{(i)}+1<n^{(i+1)}$ since for a given $n$, there are only finitely many partitions $\lambda$ for which $\lambda[n]$ is defined. Now construct $(F_n)_n$ by letting $F_n=0$ unless $n=n^{(i)}$ in which case we let $F_n=f_{\lambda^{(i)}[n^{(i)}]}$. The sequence is $f$-stable, since any fixed $\lambda$ shows up in the sequence $(\lambda^{(i)})_i$ at most once and so
the coefficient of  $f_{\lambda[n]}$ is eventually 0. However, the coefficient of $g_{\mu[n]}$ oscillates between being zero and non-zero by the choice of $n^{(i)}$. This contradicts stability transfer implying the necessity of condition 2.\parspace

Now we argue that the conditions are sufficient for stability transfer. Let $(F_n)_n$ be an $f$-stable sequence with expansion in $f$
\begin{align*}
F_n&=\sum_{\lambda}d_{\lambda,n}f_{\lambda[n]}=\sum_{\mu}\p{\sum_{\lambda\in J^{f\rightarrow g}_\mu}d_{\lambda,n}A_{\lambda[n],\mu[n]}}g_{\mu[n]},
\end{align*}
Then coefficient of $g_{\mu[n]}$,
$\sum_{\lambda\in J^{f\rightarrow g}_\mu}d_{\lambda,n}A_{\lambda[n],\mu[n]}$, stabilizes once $n\geq \rg^f_\lambda(F_\bullet),~N^{f\rightarrow g}_{\lambda,\mu}$ for all $\lambda\in J^{f\rightarrow g}_{\mu}$. Hence
\begin{equation}\label{12:ineq}
    \rg^g_\mu(F_\bullet)\leq \max\left\{\max\p{\rg^f_\lambda(F_\bullet),~N^{f\rightarrow g}_{\lambda,\mu}}:~\lambda\in J^{f\rightarrow g}_\mu\right\}<\infty
\end{equation}

implying that $F_\bullet$ is $g$-stable.
\newline\newline
\textit{(uniform stability transfer from $f$ to $g$)}: The necessity of condition 1 follows from the weak stability transfer case. To see the necessity of condition 3, take the uniformally $f$-stable sequence from the first case. To have uniform $g$-stability, the coefficient of $g_{\mu[n]}$, $A_{\lambda[n],\mu[n]}$, must be zero for all but finite $\mu$. But this also implies that for all but finite $\mu$, $A_{\lambda[n],\mu[n]}=0$ for all $n\geq 0$.\parspace

In the other direction, we have that if $F_\bullet$ is uniformally $f$-stable, then condition 1 implies that $F_\bullet$ is $g$-stable and condition 3 implies that $F_\bullet$ has finite $g$-weight. Hence, $F_\bullet$ is uniformally $g$-stable and by putting inequalities (\ref{3:ineq}) and (\ref{1:ineq}) together, we get
\begin{equation}\label{13:ineq}
    \rg^g(F_\bullet)\leq \max\p{\rg^f(F_\bullet),~\max\left\{N^{f\rightarrow g}_{\lambda,\mu}:~\mu\in I^{f\rightarrow g}_{\lambda},~|\lambda|\leq \wt_f(F_\bullet)\right\}}<\infty.
\end{equation}    
\end{proof}

The dual nature of conditions 2 and 3 is captured in this following corollary.
\begin{cor}\label{23duality}
    Suppose we have a pair of homogeneous bases $(f_\lambda)_\lambda$, $(g_\lambda)_\lambda$ with corresponding dual bases $(f^*_\lambda)_\lambda,(g^*_\lambda)_\lambda$. Then the following are equivalent:
\begin{itemize}
    \item Stability transfers from $f$ to $g$
    \item Uniform stability transfers from $g^*$ to $f^*$
\end{itemize}
\end{cor}

\begin{proof}
     Let $A=(A_{\lambda,\mu})_\mu$ be the change of basis matrix from $g$ to $f$: $f_\lambda=\sum_\mu A_{\lambda,\mu}g_\mu$. Then the adjoint $A^T$ gives the change of basis 
$$g^*_\mu=\sum_{\lambda} A^T_{\mu,\lambda}f^*_\lambda=\sum_{\lambda}A_{\lambda,\mu}f^*_\lambda.$$
It is clear that condition 1 holds for $A$ if and only if it holds for $A^T$. The claim follows since condition 2 for the $A_{\lambda[n],\mu[n]}$ is equivalent to condition 3 holding for $A^T_{\lambda[n],\mu[n]}=A_{\mu[n],\lambda[n]}$. 
\end{proof}

\subsection{Application to the classical bases}

In this section, we consider the classical symmetric function basis, $m_\lambda,s_\lambda,h_\lambda,e_\lambda,p_\lambda,p_\lambda/z_\lambda$ and determine which of conditions 1,2,3 from Theorem \ref{stabTransferResult} hold for each ordered pair. Table \ref{basisTable} summarizes these computations.

\begin{prop}\label{genTransfer}
    Let $F=(F_n\in \Lambda^{(n)})_n$. Then the following statements hold:
    \begin{itemize}
        \item $F$ is $s$-stable if and only if it is $m$-stable
        \item $F$ is uniformally $s$-stable if and only if it is uniformally $h$-stable (dual to previous statement by Corollary \ref{23duality})
        \item $F$ is $h$-stable if and only if it is $p/z$-stable
        \item $F$ is uniformally $m$-stable if and only if it is uniformally $p$-stable (dual to previous statement by Corollary \ref{23duality})
    \end{itemize}

    We have bounds on the following relevant $N^{f\rightarrow g}_{\lambda,\mu}$ values for when $|\lambda|\leq |\mu|$:
    \begin{align}\label{Nbounds}
        N^{s\rightarrow m}_{\lambda,\mu},~N^{m\rightarrow s}_{\lambda,\mu}&\leq 2|\mu|,\\
        N^{p/z\rightarrow h}_{\lambda,\mu},~N^{h\rightarrow p/z}_{\lambda,\mu}&\leq 2|\mu|+1
    \end{align}
    as well as the following values for the $J^{f\rightarrow g}_{\mu}$ and $I^{f\rightarrow g}_{\lambda}$
    \begin{align}
        J^{s\rightarrow m}_{\mu},~J^{m\rightarrow s}_{\mu},J^{p/z\rightarrow h}_{\mu},~J^{h\rightarrow p/z}_{\mu}&\subseteq\{\lambda:~|\lambda|\leq |\mu|\}\label{Jbounds}\\
        I^{p\rightarrow m}_{\lambda},~I^{m\rightarrow p}_{\lambda},I^{h\rightarrow s}_{\lambda},~I^{s\rightarrow h}_{\lambda}&\subseteq\{\mu:~|\mu|\leq |\lambda|\}\label{Ibounds}
    \end{align}
\end{prop}

\begin{table}[ht!]
\centering
\begin{tabular}{|c|c|c|c|c|c|c|}
\hline
 & $m_{\lambda}$ & $s_{\lambda}$ & $h_{\lambda}$ & $e_{\lambda}$ & $p_{\lambda}$ & $p_{\lambda}/z_{\lambda}$ \\ \hline
 
$m_{\mu}$ &  & $K_{\lambda,\mu}$ : 1,2 & $IM_{\lambda,\mu}$ : 1 & $BM_{\lambda,\mu}$ : 1 & $|OB_{\lambda,\mu}|$ : 1,3 & $../z_\lambda$ : 3 \\ \hline

$s_{\mu}$ & $K^{-1}_{\lambda,\mu}$ : 1,2 &  & $K_{\mu,\lambda}$ : 1,3 & $K_{\mu',\lambda}$ : 1 &  $\chi^\mu_\lambda$ : 1 & $../z_\lambda$ : - \\ \hline

$h_{\mu}$ & $IM^{-1}_{\lambda,\mu}$ : - & $K^{-1}_{\mu,\lambda}$ : 1,3 &  & \makecell{$(-1)^{n-l(\mu)}|B_{\mu,\lambda}|$ \\
: 2} & \makecell{$(-1)^{l(\lambda)-l(\mu)}w(B_{\mu,\lambda})$ \\
: 2} & $../z_\lambda$ : 1,2 \\ \hline

$e_{\mu}$ & $BM^{-1}_{\lambda,\mu}$ : - & $K^{-1}_{\mu,\lambda'}$ : - & \makecell{$(-1)^{n-l(\mu)}|B_{\mu,\lambda}|$\\
: 2} &  & \makecell{$(-1)^{n-l(\mu)}w(B_{\mu,\lambda})$\\
: 2} & $../z_\lambda$ : 2\\ \hline

$p_{\mu}$ & \makecell{$(-1)^{l(\mu)-l(\lambda)}\frac{w(B_{\lambda,\mu})}{z_\mu}$\\ : 1,3} & $\frac{\chi^\lambda_\mu}{z_\mu}$ : - & $\frac{|OB_{\mu,\lambda}|}{z_\mu}$ : 2 & \makecell{$(-1)^{n-l(\mu)}\frac{|OB_{\mu,\lambda}|}{z_\mu}$\\
: 2} &  & $\frac{\delta_{\lambda,\mu}}{z_\lambda}$ : 2,3 \\ \hline

$p_{\mu}/z_{\mu}$ & \makecell{$(-1)^{l(\mu)-l(\lambda)}w(B_{\lambda,\mu})$ \\
: 3}  & $\chi^\lambda_\mu$ : 1 & $|OB_{\mu,\lambda}|$ : 1,2 & \makecell{$(-1)^{n-l(\mu)}|OB_{\mu,\lambda}|$\\
: 2} & $\delta_{\lambda,\mu}z_\mu$ : 2,3 &  \\ \hline
\end{tabular}
\caption{Verification of which conditions from Theorem \ref{stabTransferResult} hold for pairs of classical symmetric function basis. The $(f_\lambda,g_\mu)$-entry indicates the change of basis coefficients from $f_\lambda=\sum_\mu A_{\lambda,\mu}g_\mu$ as well as listing which of the conditions hold. We record - if all conditions 1,2,3 fail. Diagonal entries are left blank.}
\label{basisTable}
\end{table}

As a warm up, we will analyze the Kostka and inverse Kostka numbers to show (\ref{Nbounds}-\ref{Ibounds}) for change of basis between the Schur basis and the monomial basis.

\begin{lem}\label{kostkaStable}
Once $n\geq  2|\mu|$
$$K_{\lambda[n],\mu[n]}=\begin{cases}\langle s_\lambda h_{|\mu|-|\lambda|},h_\mu\rangle,&|\mu|\geq |\lambda|\\
0,&\text{else}\end{cases}$$
and in fact the bound is sharp in the sense that
$$\max_{\lambda}(N^{s\rightarrow m}_{\lambda,\mu})=2|\mu|$$
We have that $J^{s\rightarrow m}_\mu\subseteq \{\lambda:~|\lambda|\leq |\mu|\}$, but $I^{s\rightarrow m}_{\lambda}$ is infinite since it contains all $\mu=(1^k)$.\parspace

Analogously, the inverse Kostka numbers have the following stability property. We have that $K^{-1}_{\lambda[n],\mu[n]}$ stabilize once $n\geq 2|\mu|$, so
$$\max_{\lambda}(N^{m \rightarrow s}_{\lambda,\mu})\geq 2|\mu|,$$
but the bound may not be sharp. We also have that $K^{-1}_{\lambda[n],\mu[n]}=0$ is zero unless $|\mu|\geq |\lambda|$ so $J^{m\rightarrow s}_{\mu}\subseteq \{\lambda:~|\lambda|\leq |\mu|\}$ but $I^{m\rightarrow s}_{\lambda}$ is infinite since it contains all $\mu=(1^k)$.
\end{lem}

We note the similarity between the first half of the statement in Lemma \ref{kostkaStable} and Theorem 3.3 in \cite{colmenarejo}. The proof of this statement uses the same map on standard Young tableaux as defined in the proof of that theorem.

\begin{proof}
First we argue that if $|\mu|<|\lambda|$, then $K_{\lambda[n],\mu[n]}=0$ assuming $\lambda[n],\mu[n]$ are defined. We have that $n-|\mu|>n-|\lambda|$ and so the numbers of cells labeled 1 in any $T\in \text{SSYT}(\lambda[n],\mu[n])$ exceeds the size of the bottom row of $T$, which would yield a contradiction.\parspace

Thus, we consider the case $|\mu|\geq |\lambda|$. Let $K^{(n)}:=\text{SSYT}(\lambda[n],\mu[n])$. Define the following map $\kappa_n:K^{(n)}\rightarrow K^{(n+1)}$. Given a $T\in K^{(n)}$, slide the bottom row to the right by one cell and insert a new cell containing a 1 in the bottom left corner. See the following example:

$$\begin{ytableau}2&4&4\\1&1&3&3\end{ytableau}\,\longmapsto\, \begin{ytableau}2&4&4\\1&1&1&3&3\end{ytableau}$$

Clearly, the map is injective and well-defined, implying $K_{\lambda[n],\mu[n]}\leq K_{\lambda[n+1],\mu[n+1]}$. We show that the map is surjective for $n\geq 2|\mu|$. In this case,
$$n\geq |\lambda|+\lambda_1,|\mu|+\mu_1,$$
so that both $\lambda[n],\mu[n]$ will are well-defined. For any $T\in K^{(n+1)}$, the number of cells labeled 1 in $T$ is 
$$n+1-|\mu|>|\mu|\geq |\lambda|\geq \lambda_1$$
and these cells are all left justified in the bottom row. Removing the bottom left cell and sliding the remaining entries of the bottom row to the left still yields a valid tableau $T'\in K^{(n)}$ that maps to $T\in K^{(n+1)}$ since the first $\lambda_1$ cells of the bottom row of $T'$ are forced to be labeled 1 by the assumption on $n$. Hence, $K_{\lambda[n],\mu[n]}$ stabilizes for $n\geq 2|\mu|$.\parspace

For these sufficiently large $n$, one may take any
$T\in K^{(n)}$, remove all cells labeled 1 and decrement the remaining labels by 1 to get a semi-standard Young tableau of skew shape $\nu$ and content $\mu$, where $\nu$ can be viewed as disjoint union shapes $\lambda$ and $(|\mu|-|\lambda|)$. This bijection between elements of $K^{(n)}$ and skew tableau implies $K_{\lambda[n],\mu[n]}=K_{\nu,\mu}$, a skew-Kostka number. Then
$$K_{\lambda[n],\mu[n]}=K_{\nu,\mu}=\langle s_\nu,h_\mu\rangle=\langle s_\lambda h_{|\mu|-|\lambda|},h_\mu\rangle.$$
To see the sharpness of the bound take any $\mu$ and let $\lambda=(|\mu|)$. Then the stable value is $\langle s_\lambda,h_\mu\rangle =K_{(|\mu|),\mu}=1$ non-zero and requires $\lambda[n]$ to be defined, i.e. $n\geq |\lambda|+\lambda_1=2|\mu|$.\parspace

Now we focus on the inverse Kostka numbers, let $L^{(n)}:=\text{SRHT}(\mu[n],\lambda[n])$ be the set of special rim hook fillings of shape $\mu[n]$ and content $\lambda[n]$, so that
$$K^{-1}_{\lambda[n],\mu[n]}=\sum_{T\in L^{(n)}}\sgn(T).$$
If $|\mu|<|\lambda|$, then no special rim hook fillings exist. This is because there needs to exist a rim hook containing both a cell in the first row and the rightmost cell in the bottom row of $\mu[n]$. Hence, there has to exist a rim hook of length at least $n-|\mu|$. But the longest hook will have length $n-|\lambda|<n-|\mu|$.\parspace

Otherwise, we have $|\lambda|\leq |\mu|$. If $n\geq 2|\mu|$, then we have $n-|\mu|\geq |\mu|\geq \lambda_1$ which leads to two cases. For the first case, let $n-|\mu|>\lambda_1$.
We construct a map
$\phi_n:L^{(n)}\rightarrow L^{(n+1)}$ that extends the rim hook containing the rightmost cell of the bottom row, $c$, to contain the newly added cell, $c'
=\mu[n+1]/ \mu[n]$. Note that this map is clearly injective and preserves sign since no vertical crossings are introduced. See Fig \ref{invKostkaMap}.\parspace

\begin{figure}[ht!]
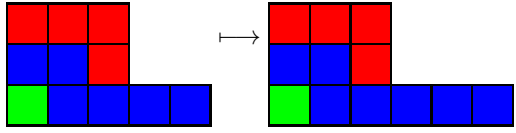

    \centering
    $$\begin{ytableau}
*(red) & *(red)  &*(red)  \\
*(blue)  & *(blue)  &*(red)\\
*(green)& *(blue)& *(blue)& *(blue)& *(blue)
\end{ytableau}\longmapsto \begin{ytableau}
*(red) & *(red)  &*(red)  \\
*(blue)  & *(blue)  &*(red)\\
*(green)& *(blue)& *(blue)& *(blue)& *(blue)& *(blue)
\end{ytableau}$$
    \caption{An application of $\phi_{11}$ to an element of $L^{(11)}$ with $\lambda[11]=(6,4,1)$ and $\mu[11]=(5,3,3)$. The sign of both special rim hook tableau is $-1$.}
    \label{invKostkaMap}
\end{figure}

We must check that the map is well-defined, i.e. that the content of the new special rim hook tableau is $\lambda[n+1]$. Within $T\in L^{(n)}$, only the special rim hook of length $n-|\lambda|$ is long enough to contain $c$ since the second longest hook has length $\lambda_1< n-|\mu|$. Thus $c'$ is added to the hook of length $n-|\lambda|$ yielding a special rim hook tableau with content $\lambda[n+1]$. It easily follows that the map, $\phi_n$, is surjective and so
$$K^{-1}_{\lambda[n],\mu[n]}=\sum_{T\in L^{(n)}}\sgn(T)=\sum_{T'\in L^{(n+1)}}\sgn(T')=K^{-1}_{\lambda[n+1],\mu[n+1]}.$$

In the remaining case $n-|\mu|=\lambda_1$, which combined with
$n-|\mu|\geq |\mu|\geq \lambda_1$ implies that $n=2|\mu|$.
Additionally, $|\lambda|\leq |\mu|$ implies $|\lambda|=\lambda_1$ and $\lambda[n]=(n-|\lambda|,\lambda_1)=(|\mu|,|\mu|)$. Then a special rim hook filling of shape $\mu[n+a]$ and content $(|\mu|+a,|\mu|)$ will have a hook of length $|\mu|+a$ occupying precisely the cells of the bottom row and $K^{-1}_{\lambda[n+a],\mu[n+a]}=K^{-1}_{(|\mu|),\mu}$, independent of $a\geq 0$.
\end{proof}

\begin{rem}
    It is possible for $\max_{\lambda}(N^{m\rightarrow s}_{\lambda,\mu})<2|\mu|$. For example one may take $\mu=(2,2)$ and verify that all $K^{-1}_{\lambda[n],\mu[n]}$ stabilize once $n\geq 7$. If $\mu$ is a hook shape however, then we do get $\max_{\lambda}(N^{m\rightarrow s}_{\lambda,\mu})=2|\mu|$ by considering $\lambda=(|\mu|)$.
\end{rem}

\begin{proof}[Proof of Proposition \ref{genTransfer} and verification of table entries]
We begin by handling the remaining entries of Table \ref{basisTable} involving Kostka or inverse Kostka numbers. The coefficients for the change of basis $e_{\lambda[n]}\rightarrow s_{\mu[n]}$, $K_{\mu[n]',\lambda[n]}$, are zero for sufficiently large $n$, but $K_{(1^n)',(n)}=1$ and $K_{(n)',(1^n)}=1$ implying conditions 3 and 2 fail. The inverse change of basis $s_{\lambda[n]}\rightarrow e_{\mu[n]}$, $K^{-1}_{\mu[n],\lambda[n]'}$, has the issue that $K^{-1}_{(n),(n)'}=(-1)^{n-1}$ does not stabilize and so condition 1 fails and one may also check that conditions 2 and 3 fail.\parspace

We observe that all entries having a sign of the form $(-1)^{n-l(\mu)}$ will oscillate in sign as $n$ grows and be non-zero if we let $\mu=\lambda=\emptyset$. Hence condition 1 fails for all such entries.\parspace

Now we check the entries involving brick tableaux, $B_{\mu,\lambda}$. Although a bit surprising it turns out that the change of basis coefficients for $p/z \rightarrow h$:
$$(-1)^{l(\lambda[n])-l(\mu[n])}\frac{w(B_{\mu[n],\lambda[n]})}{z_{\lambda[n]}}$$
satisfy properties 1 and 3. For the coefficient to be non-zero the largest part of size, $n-|\lambda|$, must be at least as great as the longest brick, $n-|\mu|$. Hence, the coefficient will be zero unless $|\lambda|\leq |\mu|$ and so $J^{p/z\rightarrow h}_{\mu}\subseteq \{\lambda:~|\lambda|\leq |\mu|\}$ implies that condition 3 holds. To see the stabilization of the coefficient, we take $n\geq 2|\mu|+1$, and note that for non-zero coefficients $n-|\lambda|\geq n-|\mu|>|\mu|\geq |\lambda|\geq \lambda_1$ so that $z_{\lambda[n]}=z_\lambda\cdot (n-|\lambda|)$:
\begin{align}
    (-1)^{l(\mu[n])-l(\lambda[n])}\frac{w(B_{\mu[n],\lambda[n]})}{z_{\lambda[n]}}&=(-1)^{l(\mu)-l(\lambda)}\frac{w(B_{\mu[n],\lambda[n]})}{z_{\lambda}\cdot (n-|\lambda|)}\label{p/z->hcoef}\\
    &=\frac{(-1)^{l(\mu)-l(\lambda)}}{z_{\lambda}\cdot (n-|\lambda|)}\sum_{\substack{\nu \subseteq_{\text{set}} \mu\\ |\nu|=|\mu|-|\lambda|}}w(B_{\mu\backslash \nu,\lambda})\cdot 
 w(B_{\nu\cup(n-|\mu|),(n-|\lambda|)})  \label{splitW}\\
    &=\frac{(-1)^{l(\mu)-l(\lambda)}}{z_{\lambda}\cdot (n-|\lambda|)}\sum_{\substack{\nu \subseteq_{\text{set}} \mu\\ |\nu|=|\mu|-|\lambda|}}w(B_{\mu\backslash \nu,\lambda})\cdot(n-|\lambda|)\cdot|\mathcal{R}(\nu)| \label{cycleReduced}\\
    &=\frac{(-1)^{l(\mu)-l(\lambda)}}{z_{\lambda}}\sum_{\substack{\nu \subseteq_{\text{set}} \mu\\ |\nu|=|\mu|-|\lambda|}}w(B_{\mu\backslash \nu,\lambda})\cdot|\mathcal{R}(\nu)|,\label{reducedSum}
\end{align}
where $\mathcal{R}$ denotes the set of rearrangements of parts. Line (\ref{splitW}) follows since the largest brick, $n-|\mu|>\lambda_1$, must go in the bottom row along with the remaining $|\mu|-|\lambda|$ spaces filled by bricks given by some $\nu$. It is easy to check that since the brick fillings of the bottom row and the rest of $\lambda[n]$ are independent of each other upon specifying $\nu$, allowing the $w$-sum to factor as seen in line (\ref{splitW}). To see line (\ref{cycleReduced}), we have that
$$w(B_{\nu\cup(n-|\mu|),(n-|\lambda|)})=\sum_{T\in \mathcal{R}(\nu\cup (n-|\mu|))}w(T),$$
where $w(T)$ is the size of the last part in the rearrangement, $T$. Since $n-|\mu|$ is greater than all parts of $\nu$, all cyclic rearrangements are distinct. Letting
$T \sim T'$ indicate that $T$ and $T'$ are cyclic rearrangements and
$[T]$ denote the corresponding equivilance class, we get
\begin{align*}
    \sum_{[T]\in \mathcal{R}(\nu\cup (n-|\mu|))/\sim}\p{\sum_{T'\in [T]}w(T')}&=\sum_{[T]\in \mathcal{R}(\nu\cup (n-|\mu|))/\sim}(n-|\lambda|)\\
    &=(n-|\lambda|)\cdot |\mathcal{R}(\nu\cup (n-|\mu|))/\sim|\\
    &=(n-|\lambda|)\cdot \frac{1}{l(\nu)+1}\binom{l(\nu)+1}{1,c_1(\nu),c_2(\nu),\hdots}\\
    &=(n-|\lambda|)\binom{l(\nu)}{c_1(\nu),c_2(\nu),\hdots}\\
    &=(n-|\lambda|)\cdot |\mathcal{R}(\nu)|,
\end{align*}
since all equivilance classes have size $l(\nu)+1$. Thus, condition 1 holds as well as condition 3. Note that the stable range $n\geq 2|\mu|+1$ cannot be improved to $n\geq 2|\mu|$, since we can consider $\mu=\lambda=(k)$ for any $k$. Then the coefficient (\ref{p/z->hcoef}) for $n=2k$ is
$$\frac{w(B_{(k,k),(k,k)})}{z_{(k,k)}}=\frac{k^2}{2!\cdot k^2}=\frac12,$$
while the coefficient (\ref{p/z->hcoef}) for $n=2k+1$ is
$$\frac{w(B_{(k+1,k),(k+1,k)})}{z_{(k+1,k)}}=\frac{(k+1)k}{(k+1)k}=1.$$

Also note that by letting $\lambda=\emptyset$ and $\mu=(1^{n-1})$, we see that condition 2 fails. Thus, the adjoint change of basis $m\rightarrow p$ satisfies conditions 1,2 but not 3. We also see how the factor $z_{\lambda[n]}$ is necessary for convergence and so the change of basis $p\rightarrow h$ does not have stabilizing coefficients.\parspace

To consider the reverse change of basis $h\rightarrow p/z$, we recall that $OB_{\mu,\lambda}$ can be interpreted as the set of brick tableaux of shape $\lambda$, where the bricks from $\mu$ have distinguished labels, but their order within the rows of $\lambda$ is irrelevant. As expected, $|OB_{\mu[n],\lambda[n]}|=0$ unless $|\mu|\geq |\lambda|$ and trust that the reader can check that the coefficients stabilizes once $n\geq 2|\mu|+1$, since this would force the largest brick into the bottom row. Again, we can't say in general that $n\geq 2|\mu|$ is a stable range since one can take $\mu=\lambda=(k)$, for which $|OB_{(k,k),(k,k)}|=2$, but $|OB_{(k+1,k),(k+1,k)}|=1$.\parspace

Next we consider $IM_{\lambda,\mu}$ and $BM_{\lambda,\mu}$. For sufficiently large $n$, $IM_{\lambda[n],\mu[n]}=IM_{\lambda[n+1],\mu[n+1]}$ since there is a bijection between the $(l(\lambda)+1)\times(l(\mu)+1)$ matrices being enumerated that adds a 1 to the $(1,1)$-entry. The map is surjective since this entry is non-zero for large $n$. On the other hand $BM_{\lambda[n],\mu[n]}$ is eventually zero, since the $(1,1)$-entry of any contributing matrix must eventually be $>1$, which is impossible for boolean matrices. On the other hand conditions 2 and 3 fail since $IM_{(n),(1^n)}=BM_{(n),(1^n)}=1$ and the change of basis matrices are symmetric.\parspace

As for the inverse change of basis $m\rightarrow h$ (resp. $m\rightarrow e$), we may simply use the fact that $m_{(n)}=p_{(n)}$ and the previously analyzed $p\rightarrow h$ (resp. $p\rightarrow e$) to see that condition 1 fails. One can check that conditions 2 and 3 fail since $IM^{-1}_{(1^n),(n)}=IM^{-1}_{(n),(1^n)}\neq 0$ and $BM^{-1}_{(1^n),(n)}=BM^{-1}_{(n),(1^n)}\neq 0$, which follows from the combinatorial formula for $IM^{-1},BM^{-1}$ involving primitive bi-brick tableau introduced in Theorem 1 of \cite{primBrick}.\parspace

Finally, we look at the beloved character values, $\chi^\lambda_\mu$. It is easy to see that $\chi^{\lambda[n]}_{\mu[n]}$ stabilizes since we can insist that a rim hook tableau removes the longest hook, $n-|\mu|$, first. For $n$ sufficiently large, it will be forced to use the bottom row of $\lambda[n]$ and will thus start at the bottom rightmost cell with the rest of the rim hook placement being determined. Then we can establish $\chi^{\lambda[n]}_{\mu[n]}=\chi^{\lambda[n+1]}_{\mu[n+1]}$ by using a sign-preserving bijection that extend the longest rim hook by one cell to the right. Conditions 2 and 3 fail however since $\chi^{(1^n)}_{(n)}=(-1)^{n-1}$ and $\chi^{(n)}_{(1^n)}=1$. The remaining cases follow at once.    
\end{proof}

\begin{cor}\label{wtEquality}
    All sequences $(F_n\in \Lambda^{(n)})_n$ satisfy $\wt_s(F_\bullet)=\wt_h(F_\bullet)$ and $\wt_m(F_\bullet)=\wt_p(F_\bullet)$.
\end{cor}

\begin{proof}
    By substituting (\ref{Ibounds}) for $s\rightarrow h$ into (\ref{3:ineq}), one gets $\wt_h(F_\bullet)\leq \wt_s(F_\bullet)$. By the same argument, it follows that $\wt_s(F_\bullet)\leq \wt_h(F_\bullet)$ as well as the inequalities establishing $\wt_m(F_\bullet)=\wt_p(F_\bullet)$.
\end{proof}

\begin{rem}
    It may appear as if whenever conditions 1 and 2 hold for a change of basis matrix, $A$, then they also hold for the inverse, $A^{-1}$, as seen in the biconditional statements in Proposition \ref{genTransfer}. This is just a coincidence and an artifact of the change of basis matrices being particularly nice lower triangular matrices (with respect to an appropriate basis ordering).\parspace
    
    To see that this is not true in general, fix any total ordering, $\prec$, such that $|\lambda|+\lambda_1< |\mu|+\mu_1$ implies
    $\lambda \prec\mu$. Call the partitions with respect to this ordering $\lambda^{(1)},\lambda^{(2)},\hdots$ and observe that by the choice of ordering, the partitions of $n$ are precisely the partitions in the tuple
    $$L_n=(\lambda^{(1)}[n],\lambda^{(2)}[n],\cdots \lambda^{(p(n))}[n]),$$
    where $p(n)$ is the partition number. Define a new ordering on partitions by concatenating the tuples $L=L_0L_1L_2\cdots$ and define a matrix, $A$, indexed by partitions ordered by $L$:
    $$A_{\lambda^{(i)}[n],\lambda^{(j)}[n]}:=\begin{cases}
        1,&i=j,j+1,~\text{and }\lambda^{(i)}[n],\lambda^{(j)}[n]\text{ defined}\\
        0&\text{else}.
    \end{cases}$$
    so that each restriction of rows and columns to partitions of $n$ is 
    $$A^{(n)}:=A_{L_n,L_n}=J_{p(n)}(1)^T,$$
    a $p(n)\times p(n)$ lower triangular Jordan block with eigenvalue 1. One may also view $A$ as the block diagonal concatenation
    $$A=\text{diag}(A^{(0)},A^{(1)},A^{(2)},\cdots).$$
    This matrix can be interpreted as a change of basis matrix between some homogeneous symmetric function bases with respect to the order on partitions given by $L$. The matrix is invertible since it is block lower triangular with 1's on the diagonal and the inverse can be explicitly given by inverting each Jordan block
    $$A^{-1}_{\lambda^{(i)}[n],\lambda^{(j)}[n]}:=\begin{cases}
        (-1)^{i-j},&i\geq j,~\text{and }\lambda^{(i)}[n],\lambda^{(j)}[n]\text{ defined}\\
        0&\text{else}.
    \end{cases}.$$

    It follows that $A_{\lambda[n],\mu[n]}$ and $A^{-1}_{\lambda[n],\mu[n]}$ are both stable as soon $\lambda[n],\mu[n]$ are defined, which implies that both $A$ and $A^{-1}$ satisfy condition 1. By looking at columns of the Jordan blocks, we see that $A$ satisfies condition 2 with $J_{\lambda^{(j)}}=\{\lambda^{(j)},\lambda^{(j+1)}\}$ finite. However, $A^{-1}_{\lambda^{(i)}[n],\lambda^{(j)}[n]}\neq 0$ as long as $i\geq j$ and one takes an appropriately large $n$ implying condition 2 fails.\parspace

    By taking the transpose of $A$, we get a matrix for which conditions 1,3 hold, but $A^{-1}$ fails condition 3.
\end{rem}

\subsection{Concluding uniform Schur stability}
We can now apply the inequalities (\ref{3:ineq}-\ref{13:ineq}) to get the following corollary, which is relevant for concluding Schur stability and thus can help establish representation stability.

\begin{cor}\label{schurStabFrom:m,p/z}
    (a) Let $(F_n\in \Lambda^{(n)})_n$ be a Schur-positive sequence. The sequence is uniformally $s$-stable if and only if $\wt_s(F_\bullet)<\infty$ and it is $m$-stable. In this case, one has a bound on the stable range
    \begin{equation}\label{m->s individual range}
        \rg_\lambda^s(F_\bullet)\leq \max\p{2\cdot |\lambda|,~\max_{|\mu|\leq |\lambda|}\rg^m_\mu(F_\bullet)}
    \end{equation}
    and uniform stable range
    \begin{equation}\label{m->s uniform range}
        \rg^s(F_\bullet)\leq\max\p{2\cdot \wt(V_\bullet),~\max_{|\mu|\leq \wt_s(F_\bullet)}\rg^m_\mu(F_\bullet)}.
    \end{equation}
    (b) A sequence $(F_n\in \Lambda^{(n)})_n$ is uniformally $s$-stable if and only if $\wt_s(F_\bullet)<\infty$ and it is $p/z$-stable. In this case, one has a bound on the uniform stable range
    \begin{equation}\label{p/z->s stable range}
        \rg^s(F_\bullet)\leq \max\p{2\cdot \wt_s(F_\bullet)+1,~\max_{|\mu|\leq \wt_s(F_\bullet)}\rg^{p/z}_\mu(F_\bullet)}
    \end{equation}
\end{cor}

\begin{proof}
    The forward direction for (a) is clear from Proposition \ref{genTransfer}, while the reverse direction follows from the case $m\rightarrow s$ of (\ref{Nbounds}) and (\ref{Jbounds}) being substituted into (\ref{12:ineq}).\parspace

    The forward direction for (b) is clear from the $s\rightarrow p/z$ entry of Table \ref{basisTable}. For the reverse direction, we have that $p/z$-stability of $F_\bullet$ implies $h$-stability with
    $$\rg^h_\nu(F_\bullet)\leq \max\p{2|\nu|+1,~\max_{|\mu|\leq |\nu|}\rg^{p/z}_\mu(F_\bullet)}$$
     by substituting the $p/z\rightarrow h$ case in (\ref{Nbounds}) and (\ref{Jbounds}) into (\ref{12:ineq}). By Corollary \ref{wtEquality}, $\wt_h(F_\bullet)=\wt_s(F_\bullet)<\infty$ so we have uniform $h$-stability and hence uniform $s$-stability. Then substituting (\ref{Nbounds}) and (\ref{Ibounds}) for the case $h\rightarrow s$ into (\ref{13:ineq}), we get
     $$\rg^s(F_\bullet)\leq \max\p{2\cdot \wt_s(F_\bullet)+1,~\max_{|\mu|\leq \wt_s(F_\bullet)}\rg^{p/z}_\mu(F_\bullet)}.$$
\end{proof}

We now state the main results as easy corollaries.
\begin{thm}\label{mainResult}
    (a) A sequence of $S_n$-modules $V=(V_n\in \textbf{Mod}_{S_n})_n$ is URMS if and only if $\text{wt}(V)<\infty$ and the sequence of Frobenius images $\mathcal{F}(V_\bullet)$ is $m$-stable. In this case, the multiplicity of $\mathbb{S}^{\lambda[n]}$ in $V_n$ stabilizes once
   \begin{equation}\label{schurRange}
       n\geq \max\p{2\cdot |\lambda|,\max_{|\mu|\leq |\lambda|}\rg^m_\mu(\mathcal{F}(V_\bullet))}.
   \end{equation}
    and a uniform stable range of $V_\bullet$ is given by
    \begin{equation}\label{totalRange}
        n\geq \max\p{2\cdot \text{wt}(V),\max_{|\mu|\leq \text{wt}(V)}\rg^m_\mu(\mathcal{F}(V_\bullet))}.
    \end{equation}    

    (b) A sequence of $S_n$-modules $V=(V_n\in \textbf{Mod}_{S_n})_n$ is URMS if and only if $\text{wt}(V)<\infty$ and for every $k\leq \wt(V)$ and every $\sigma \in S_k$,
    $$\chi^{V_n}(\sigma\cdot(k+1~k+2~\cdots~n))$$
    stabilizes for sufficiently large $n$. In this case, a uniform stable range of $V_\bullet$ is given by
    \begin{equation}\label{totalRange char version}
        n\geq \max\p{2\cdot \wt(V)+1,~\max_{|\mu|\leq \wt(V)}\rg^{p/z}_{\mu}(\mathcal{F}(V_\bullet))},
    \end{equation}
    where $\rg^{p/z}_{\mu}(\mathcal{F}(V_\bullet))$ is the minimum $n$ such that the character value
    $$\chi^{V_\bullet}(\sigma\cdot(|\mu|+1~|\mu|+2~\cdots~n)),$$
    has stabilized, where $\sigma$ is any element of $S_{|\mu|}$ with cycle type $\mu$.\hfill $\square$
\end{thm}

Note that by polynomiality of characters of $\mathbb{S}^{\lambda[n]}$ (see Theorem 3.3.4 in \cite{FI-module}, for instance), any URMS sequence will eventually have 
$$\chi^{V_\bullet}(\sigma\cdot(|\mu|+1~|\mu|+2~\cdots~n)),$$
stabilize for fixed $k$ and fixed $\sigma\in S_k$ since the character is determined by a polynomial in cycles of size $\leq\wt(V_\bullet)$. However, this result uses the stable range $n\geq \text{stab-deg}(V_\bullet)+\wt(V_\bullet)$ and does not extract a stable range from the character values, unlike (\ref{totalRange char version}) in our result.

\begin{rem}
    This remark is concerned with the sharpness of (\ref{schurRange}-\ref{totalRange char version}). As explained in the Appendix, one can omit the $2|\lambda|$ (resp. $2\cdot \wt(V)$) from (\ref{schurRange}) resp. (\ref{totalRange}) for series of symmetric functions arising from Frobenius images, although this will not have a significant impact on Sections \ref{coinvSec} and \ref{macSec}. Even with the $2|\lambda|$ dropped we may not have a sharp range. As an example consider 
    $$\mathcal{F}(\text{Ind}^{S_n}_{S_{\nu[n]}}(1))=h_{\nu[n]}=\sum_{\lambda}K_{\lambda[n],\nu[n]}\cdot s_{\lambda[n]}=\sum_\mu IM_{\nu[n],\mu[n]}\cdot m_\mu[n],$$
   where $IM_{\nu[n],\mu[n]}$ is stable once $n\geq |\nu|+|\mu|$ and this bound is strict. Hence for the case $\nu=(1^2)$, the Schur coefficient of $s_\lambda[n]$ when $\lambda=(1^2)$ is stable once
    $$n\geq \max_{|\mu|\leq |\lambda|}\p{|\nu|+|\mu|}=4$$
    by (\ref{m->s individual range}). However, the stable value
    $\langle F_n,s_{(1^2)[n]}\rangle=K_{(1^2)[n],(1^2)[n]}=1$
    is attained earlier at $n=3$. Notice that for $n\geq 4$, the sum $\sum_{|\mu|\leq |\lambda|}d_{\mu,n}K^{-1}_{\mu[n],\lambda[n]}$ expands as
    $$1\cdot 1+3\cdot (-1)+ 7\cdot 1+ 4\cdot (-1)=1,$$
    while the sum at $n=3$ has different terms
    $$1\cdot 1+3\cdot (-2)+6\cdot 1=1.$$
    Hence the bound (\ref{schurRange improve}) may not be sharp in general.\parspace

    Unlike the improved range in \ref{altProp}, one does need $2\cdot \wt(V)+1$ term in the uniform range (\ref{totalRange char version}). Consider the URMS sequence defined as
    $$V_n=\begin{cases}
        0,&n<4\\
        \mathbb{S}^{(2,2)}\oplus \mathbb{S}^{(3,1)}\oplus \mathbb{S}^{(4)}\oplus \p{\mathbb{S}^{(2,1^2)}}^{\oplus 2},&n=4\\
        \p{\mathbb{S}^{(n-2,1^2)}}^{\oplus 2},&n\geq 5
    \end{cases},$$
    which clearly has $\wt(V)=2$ and
    $\rg^s(\mathcal{F}(V_\bullet))=5=2\cdot \wt(V)+1$. However, one can check that the character values on the conjugacy classes, $\chi^{V_4}(C_{\mu[4]})=\chi^{V_n}(C_{\mu[n]})$, agree for $|\mu|\leq \wt(V)=2$ and $n\geq 5$. Listing the character values in a vector for $\mu=\emptyset,(1),(2),(1^2)$, we get
    $$(\chi^{V_4}(\mu[4]))_{|\mu|\leq 2}=\begin{matrix}
        \mu=\emptyset\\(1)\\(2)\\(1,1)
    \end{matrix}~~~~\overset{\mathbb{S}^{(2,2)}}{\begin{pmatrix}
        0\\-1\\2\\0
    \end{pmatrix}}+\overset{\mathbb{S}^{(3,1)}}{\begin{pmatrix}
        -1\\0\\-1\\1
    \end{pmatrix}}+\overset{\mathbb{S}^{(4)}}{\begin{pmatrix}
        1\\1\\1\\1
    \end{pmatrix}}+2\overset{\mathbb{S}^{(2,1,1)}}{\begin{pmatrix}
        1\\0\\-1\\-1
    \end{pmatrix}}=2\overset{\mathbb{S}^{(n-2,1,1)}}{\begin{pmatrix}
        1\\0\\0\\0
    \end{pmatrix}}=(\chi^{V_n}(\mu[n]))_{|\mu|\leq 2},$$
    where the $n$ on the right hand side is $\geq 5$. Hence, it is not sufficient to just have character values stabilizing in order to conclude that Schur multiplicities have stabilized.
\end{rem}

We conclude the section by focusing Theorem \ref{mainResult} on the particular case of torsion-free \textbf{FI}-modules.

\begin{thm}\label{mainResultTorFree}
A sequence $(V_n,\phi_n)_n$ arising from a torsion-free \textbf{FI}-module is URS, if and only if, $\dim(V_n)=O(n^k)$ for some $k$ and for each partition $\mu$, the corresponding monomial coefficient is bounded from above:
$$\sup_{n}~\langle \mathcal{F}(V_n),h_{\mu[n]}\rangle<\infty.$$
Furthermore, it suffices to check that these upper bounds exist for $|\mu|\leq k$.
\end{thm}

\begin{proof}
    The forward direction is obvious by the previous theorem. For the reverse direction, we have by Lemma \ref{weightDegree}, $\wt(V)\leq k$. Then $(V_n,\phi_n)_n$ will be URMS (and hence URS by Corollary \ref{URMS->URS}) if we can show that the sequences $(\mathcal{F}(V_n),h_{\mu[n]}\rangle)_n$ stabilize for $|\mu|\leq k$. By Lemma \ref{monomialMonotonicity}, these sequences are non-decreasing sequences of natural numbers, and so they will stabilize if and only if they are bounded from above.
\end{proof}

\section{Applications to $r$-diagonal coinvariant algebras}\label{coinvSec}
In this section, we consider sequences of multi-graded $S_n$-representations
called $r$-diagonal coinvariant algebras, the representation stability of which was studied in \cite{FI-module}. See the exposition of Bergeron \cite{bergeron} for more details on the combinatorics of these spaces.\parspace

Let $X^{(i)}_n$ be the variable set $x^{(i)}_1,\hdots x^{(i)}_n$ for $i=1,\cdots,r$. The algebra, $\C[X_n^{(1)},\hdots,X_n^{(r)}]$, has $r$-multi-grading given by degree, $\deg(x_k^{(i)})=1$. Let $S_n$ act on $\C[X_n^{(1)},\cdots,X_n^{(r)}]$ via the diagonal action by permuting each set of variables simultaneously, i.e. $\sigma\cdot x^{(i)}_k=x^{(i)}_{\sigma(k)}$. Let $\C[X_n^{(1)},\hdots,X_n^{(r)}]^{S_n}_+$ be the subalgebra of invariants under this action with no constant term. The ideal generated by this subalgebra, $\p{\C[X_n^{(1)},\hdots,X_n^{(r)}]^{S_n}_+}$, is often called the Hilbert ideal for the action. Then the $r$-\textbf{diagonal coinvariant algebra} is defined by taking the quotient by the Hilbert ideal,
$$R_n^{(r)}:=\C[X_n^{(1)},\hdots,X_n^{(r)}]/\p{\C[X_n^{(1)},\hdots,X_n^{(r)}]^{S_n}_+}.$$
Since the Hilbert ideal is homogeneous, this algebra inherits the multi-grading of $\C[X_n^{(1)},\hdots,X_n^{(r)}]$ and since the Hilbert ideal is $S_n$-invariant, we get an induced $S_n$-action on $R_n^{(r)}$. Note that if $J=(j_1,\hdots,j_r)$, the multi-graded component $[R_n^{(r)}]_J$ is also a $S_n$-representation, since the action preserves grading. \parspace

The representation theory of the diagonal coinvariant algebras is notoriously difficult. The classical case $r=1$, has been well understood (see \cite{bergeron} for an exposition) and there are combinatorial descriptions of the distribution of multiplicities of Specht modules throughout the graded components (See Eq (\ref{coinvFrob})). For the case $r=2$, there is a monomial formula (see \cite{shuffle}) for the graded Frobenius image but no combinatorial description of the Schur expansion, making this case perfect for our methods. For the case $r\geq 3$, there are no proven formulas for the Frobenius image.\parspace

Note that the representations $R_n^{(r)}$ do not admit an \textbf{FI}-module structure but rather a \textbf{co-FI}-module structure. In any case, we can still talk about the properties like weight and being URMS when considering sequences of $S_n$-representations for varying $n$. Note that while $R_n^{(r)}$ is not URMS, for any fixed $J=(j_1,\hdots,j_r)$, the graded component $[R_n^{(r)}]_J$ is in fact URMS, a celebrated result of \cite{FI-module}.\parspace

At the end of section 5.1 in \cite{FI-module}, the authors were able to use the \textbf{co-FI}-module structure of $R_n^{(r)}$ to prove the following statement which is of crucial importance to us in this section.
\begin{prop}[\cite{FI-module}]
    For any fixed $J=(j_1,\hdots,j_r)$, the sequence in $n$, $[R_n^{(r)}]_J$ has weight bounded above by $|J|:=j_1+\cdots+j_r$.
    \label{coinvWeight}
\end{prop}
By also computing a bound on the stability-degree the authors were able to conclude
\begin{prop}[\cite{FI-module}]\label{coinvFarbRange}
    For any fixed $J=(j_1,\hdots,j_n)$, the sequence in $n$, $[R_n^{(r)}]_J$ is URMS with URMS stable range
    $$n\geq 2|J|$$
\end{prop}

This stable range as well as other types of ``stable ranges'' not mentioned in this paper are studied in greater depth for the $r$-diagonal coinvariant algebras in \cite{bahran}. We study the cases $r=1,2$ through the monomial expansions and use Proposition \ref{coinvWeight} to get stable ranges without needing to compute stability-degree as done in \cite{FI-module}. We also get a refinement of Proposition \ref{coinvFarbRange} in the case $r=2$.

\subsection{Harmonics as \textbf{FI}-modules}
We recall the definition of the $r$-diagonal harmonics (see Bergeron \cite{bergeron} for exposition and proofs). Let $A_r$ denote the set of $r$-tuples, $(a_1,\cdots,a_r)$, such that $a_i\in \N$ and at least one $a_i\neq 0$. The $r$-diagonal harmonics, 
$$H_n^{(r)}:=\left\{f\in \C[X_n^{(1)},\hdots,X_n^{(r)}]:~\sum_{k=1}^n \p{\frac{\partial}{\partial_{x^{(1)}_k}}}^{a_1}\cdots \p{\frac{\partial}{\partial_{x^{(r)}_k}}}^{a_r}f=0,~(a_1,\cdots,a_r)\in A_r\right\}.$$
By the symmetry of the definition, $H_n^{(r)}$ is invariant under the diagonal action of $S_n$ and hence forms a multi-graded $S_n$-submodule of $\C[X_n^{(1)},\hdots,X_n^{(r)}]$. It is well known that the quotient map 
$$f\mapsto f+ (\C[X_n^{(1)},\hdots,X_n^{(r)}]^+_{S_n})$$
from $H_n^{(r)}$ to $R^{(r)}_n$ is a multi-grading preserving $S_n$-equivariant map and in fact
$$H_n^{(r)}\cong R^{(r)}_n$$
as multi-graded $S_n$-modules. Hence, for any $J=(j_1,\cdots,j_r)$, the $S_n$-modules $[H_n^{(r)}]_J$ form a URMS sequence with weight $\leq |J|$ and stable range $n\geq 2|J|$.\parspace

Since the $R^{(r)}_n$ are defined via quotients, they have a \textbf{co-FI}-module structure but no natural \textbf{FI}-structure. However, the harmonics, $H_n^{(r)}$, will have a natural \textbf{FI}-module structure defined as follows. Let 
$$i_n: \C[X_n^{(1)},\hdots,X_n^{(r)}]\rightarrow  \C[X_{n+1}^{(1)},\hdots,X_{n+1}^{(r)}]$$
be the standard injective homomorphism sending $x^{(i)}_k\mapsto x^{(i)}_k$ for $i=1,\hdots,r$ and $k=1,\hdots, n$. These maps are clearly $S_n$-homomorphisms satisfying $(n+1~n+2)\cdot i_{n+1}\circ i_n=i_{n+1}\circ i_n$. From this it follows that the sequence of polynomial algebras, $\C[X_{n+1}^{(1)},\hdots,X_{n+1}^{(r)}]$, arises from a torsion-free \textbf{FI}-module. Then the $H_n^{(r)}$ will correspond to an \textbf{FI}-submodule since one can easily check that
$$i_n(H_n^{(r)})\subseteq H_{n+1}^{(r)}$$
follows immediately from the definition. Hence, we obtain the following corollary of Corollary \ref{URMS->URS} and Proposition \ref{coinvFarbRange}.
\begin{cor}
    For any $J=(j_1,\cdots,j_r)$, the sequence $[H_n^{(r)}]_J$ with connective maps $i_n$ arises from a finitely generated torsion-free \textbf{FI}-module with weight $\leq |J|$ and a URS stable range $n\geq 2|J|$.]\hfill $\square$
\end{cor}

\subsection{$r=1$: the classical coinvariant algebra}
We will denote $R^{(1)}_n$ simply by $R_{n}$ and let $q$ track the degree. It is well-known (see section 8.5 of \cite{bergeron} and (1.89) in \cite{haglund}) that the graded Frobenius image of $R_n$ is given by
\begin{equation}\label{coinvFrob}
    \mathcal{F}(R_n;q)=\sum_{\lambda[n] \vdash n}\p{\sum_{T\in\text{SYT}(\lambda[n])} q^{\text{maj}(T)}}s_{\lambda[n]}=\sum_{\mu[n]\vdash n}\begin{bmatrix}
    n\\ \mu[n]
\end{bmatrix}_q
m_{\mu[n]},
\end{equation}
where $\begin{bmatrix}
    n\\ \mu[n]
\end{bmatrix}_q$ is the $q$-multinomial coefficient and $\text{maj}$ is the major index statistic on standard Young tableaux
$$\text{maj}(T)=\sum_{j\text{ below }j+1\text{ in }T}j.$$
Note that the sequence $(R_n)_n$ is not URMS since 
$$\langle\mathcal{F}(R_n),s_{\lambda[n]}\rangle=|\text{SYT}(\lambda[n])|\rightarrow \infty$$
for $\lambda\neq \varnothing$.\parspace

We will see how $[R_n]_i$ is URMS upon taking the $i$-th graded component. The stability of the multiplicities of the Specht modules in $[R_n]_i$ can be seen directly
by looking at
$$\langle \mathcal{F}([R_n]_i),s_{\lambda[n]}\rangle=\#\{T\in \text{SYT}(\lambda[n]):~\text{maj}(T)=i\},$$
and as seen in \cite{FI-module}, the map $\{T\in \text{SYT}(\lambda[n]):\text{maj}(T)=i\}\rightarrow \{T\in \text{SYT}(\lambda[n+1]):\text{maj}(T)=i\}$ that attaches a cell labeled $n+1$ to the end of a $T\in \text{SYT}(\lambda[n])$ becomes a bijection for $n\geq i+|\lambda|$. The weight bound in Proposition \ref{coinvWeight} implies that only $\lambda$ with $|\lambda|\leq i$ contribute to $[R_n]_i$. Hence, the corresponding degree-$i$ harmonics, $[H_n]_i$, form a uniformally representation stable sequence with URS stable range $n\geq 2i$.\parspace

Alternatively, we may look at the monomial expansion and apply our result, Theorem \ref{mainResult}. The monomial coefficients are given by
$$\langle \mathcal{F}([R_n]_i),h_{\mu[n]}\rangle=[q^i]\begin{bmatrix}
    n\\ \mu[n]
\end{bmatrix}_q=\#\{w\in\mathcal{R}(1^{n-|\mu|},2^{\mu_1},\hdots,(l+1)^{\mu_l}):\text{inv}(w)=i\},$$
where $[q^i]$ denotes taking the coefficient of $q^i$ and $l=l(\mu)$. The set on the right hand side consists of words with content $1^{n-|\mu|},2^{\mu_1},\hdots,(l+1)^{\mu_l}$ and $i$ inversions. Now consider the map 
$$\{w\in\mathcal{R}(1^{n-|\mu|},2^{\mu_1},\hdots,(l+1)^{\mu_l}):\text{inv}(w)=i\}\rightarrow \{w\in\mathcal{R}(1^{n+1-|\mu|},2^{\mu_1},\hdots,(l+1)^{\mu_l}):\text{inv}(w)=i\}$$
that maps $w\mapsto 1w$. We claim that for $n\geq |\mu|+i$ and $n\geq |\mu|+\mu_1$, this map is a bijection, i.e. that
$$\rg^m_{\mu}(\mathcal{F}([R_n]_i)\leq |\mu|+\max(i,\mu_1).$$
To see that it is surjective, we observe that any $w\in\mathcal{R}(1^{n+1-|\mu|},2^{\mu_1},\hdots,)$ with $\text{inv}(w)=i$, has $w_1=1$. If $w_1\neq 1$, then the initial entry, $w_1$, would contribute at least $n+1-|\mu|>i$ inversions and so we can't have $\text{inv}(w)=i$.\parspace

Since the weight of $[R_n]_i$ is bounded by $i$, we have a stable range from Eq (\ref{schurRange}) given by
$$n\geq \max\p{2i,\max_{|\mu|\leq i}|\mu|+\max(i,\mu_1)}=2i,$$
which agrees with the one given in \cite{FI-module} using their stability-degree formula as well as the explicit Schur expansion.

\subsection{$r=2$: classical diagonal coinvariant algebra}
We will denote $R_n^{(2)}$ by $DR_n$ and $H_n^{(2)}$ by $DH_n$ throughout this section and refer to the algebra as the diagonal coinvariant algebra. From Proposition \ref{coinvFarbRange}, the bi-graded components, $[DH_n]_{(i,j)}$, form a URS sequence for fixed $i,j$ with URS stable range $n\geq 2(i+j)$. Our goal in this subsection is to reproduce this stable range by looking at the monomial expansion of corresponding Frobenius image, $\mathcal{F}([DR_n]_{(i,j)})=\mathcal{F}([DH_n]_{(i,j)})$.\parspace

Letting $q$ track the $X^{(1)}_n$ degree and $t$ track the $X^{(2)}_n$ degree, the bi-graded Frobenius image is known to be 
$$\mathcal{F}(DR_n;q,t)=\nabla e_n,$$
where $\nabla$ is a certain operator defined on the modified Macdonald polynomials, $\tilde{H}_\mu[X;q,t]$, see \cite{Haiman_2002} and Theorem 2.17.1 \cite{haglund}. While it is still an open problem to find a combinatorial formula for the Schur expansion of $\nabla e_n$,
Carlsson and Mellit \cite{shuffle} proved a combinatorial formula for the monomial expansion of $\nabla e_n$. For any composition $\alpha \vDash n$, they show
\begin{equation}\label{shuffleThm}
    \langle \nabla e_n,h_\alpha\rangle =\sum_{\substack{P\in P_n\\ P~\alpha\text{-shuffle}}}q^{\text{dinv}(P)}t^{\text{area}(P)},
\end{equation}
where all the relevant terms will be subsequently defined. In particular, when $\alpha$ is of partition shape, the above yields a formula for the monomial coefficient of $m_\alpha$.\parspace

We now go through what all the terms in (\ref{shuffleThm}) mean. The set $P_n$ denotes the set of \textbf{labeled Dyck paths} in an $n$-by-$n$ grid. Recall that a Dyck path is a lattice walk in an $n$-by-$n$ grid consisting of $(0,1)$ and $(1,0)$ as steps (east and north steps),
starting at $(0,0)$, ending at $(n,n)$, and staying weakly above the diagonal $y=x$. We will encode Dyck paths as words in letters $N,E$ to denote the north (resp. east) steps taken in the lattice walk.
For example the red path in Fig \ref{labeled Dyck path} as an example is the Dyck path $NNEENNENEENE$.\parspace

A \textit{labeled} Dyck path consists of a Dyck path and a bijection $f$ from the set of north steps of the path to $[n]$, such that runs of consecutive north steps are mapped to increasing $i\in [n]$. Visually this can be represented by drawling labels to the right of each north step so that labels are increasing up columns (see Fig \ref{labeled Dyck path}).\parspace

The \textit{reading word} of a labeled Dyck path is obtained by scanning southwest along diagonals, starting from the diagonal $y=x+n$, then scanning $y=x+n-1$ and so on. Recording the labels encountered while scanning yields the reading word $\sigma(P)\in S_n$. A permutation, $\sigma$, is an $\alpha$\textbf{-shuffle} if the values $1,2,\cdots,\alpha_1$ occur in the same relative order in the one-line notation of $\sigma$, the values $\alpha_1+1,\cdots, \alpha_1+\alpha_2$ occur in the same relative order, and so on. We say $P$ is an $\alpha$-shuffle if $\sigma(P)$ is an $\alpha$-shuffle.\parspace

Finally, we describe the statistics, \textbf{dinv} and \textbf{area}, on the $P\in P_n$. The area statistic simply counts the number of cells that are \textit{entirely} contained between the Dyck path and diagonal $y=x$. The dinv statistic counts the number of \textit{diagonal inversions}. A diagonal inversion of a labeled Dyck path is a pair of cells, $b,c$, immediately right of a north step, with corresponding labels $i,j$, such that one of two following conditions are satisfied. Suppose that $b$ has the greater label, i.e. $i>j$. The first condition is that $b$ and $c$ are in the same diagonal $y=x+s$ and $b$ is further north than $c$. The second alternative condition is that $c$ is on a diagonal of the form $y=x+s$, $b$ is in a diagonal $y=x+s+1$, and $b$ is further south than $c$.\parspace

\begin{figure}[ht!]
    \centering
    \begin{tikzpicture}
    \draw[very thin, gray] (0,0) grid (6,6);
    
    \coordinate (start) at (0,0);

    \node at (0.5,0.5) {2};
    \node at (0.5,1.5) {3};
    \node at (2.5,2.5) {1};
    \node at (2.5,3.5) {4};
    \node at (3.5,4.5) {5};
    \node at (5.5,5.5) {6};
    
    \draw[thick, red, -] 
        (start) 
        -- ++(0,1)   
        -- ++(0,1)   
        -- ++(1,0)   
        -- ++(1,0)   
        -- ++(0,1)   
        -- ++(0,1)   
        -- ++(1,0)   
        -- ++(0,1)   
        -- ++(1,0)   
        -- ++(1,0)   
        -- ++(0,1)   
        -- ++(1,0) ;  
 \end{tikzpicture}
    \caption{A labeled Dyck path, $P$, with reading word $\sigma(P)=543612$, which is a $\alpha$-shuffle for $\alpha=(2,1,1,2)$ as well as all refinements of this composition. The area of $P$ is 3 and $\text{dinv}(P)=6$ with the following label pairs contributing diagonal inversions: $(2,6),(1,6),(3,4),(4,5),(3,5),(1,3)$. Hence $P$ contributes to the monomial coefficients of $m_{(2^2,1^2)},\,m_{(2,1^3)},\,m_{(1^4)}$ in $[DR_6]_{(5,3)}$.}
    \label{labeled Dyck path}
\end{figure}
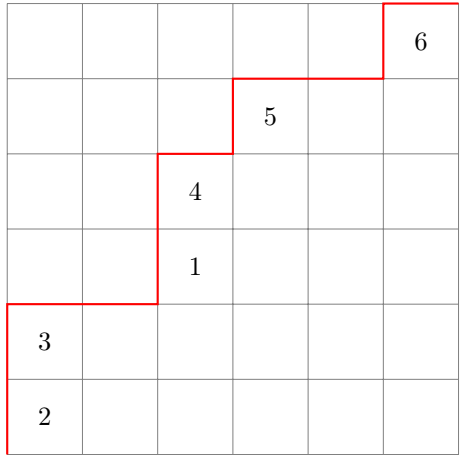

We now state the main result of this subsection. The following proposition may be viewed a refinement of Proposition \ref{coinvFarbRange} of Church, Farb, and Ellenberg, since it finds individual stable ranges for the multiplicities of Specht modules in $[DR_n]_{(i,j)}$. 

\begin{prop}\label{DRnStability}
       The coefficient of $\mathbb{S}^{\lambda[n]}$ in $[DR_n]_{(i,j)}\cong_{S_n} [DH_n]_{(i,j)}$ stabilizes once $n\geq \max(|\lambda|+i+j,2|\lambda|)$.
\end{prop}

\begin{proof}
For our analysis, it will be convenient to use a rearrangement of the parts of $\mu[n]$. Define the composition $\alpha^{(n)}=(\mu_1,\cdots,\mu_l,n-|\mu|)$ for $n\geq |\mu|+\mu_1$. Then (\ref{shuffleThm}) tells us that
\begin{align}
    \left\langle \mathcal{F}\p{[DR_n]_{(i,j)}},h_{\mu[n]}\right\rangle &= [q^it^j]\left\langle \mathcal{F}(DR_n;q,t),h_{\mu[n]}\right\rangle\\
    &=[q^it^j]\left\langle \mathcal{F}(DR_n;q,t),h_{\alpha^{(n)}}\right\rangle\\
    &=\#\{P\in P_n:~P\text{ is an }\alpha^{(n)}\text{-shuffle},~\text{dinv}(P)=i,~\text{area}(P)=j\}\label{Dn}
\end{align}

Our goal is to find a range $n\geq M_\mu$ for which the above quantity has stabilized as this will allow us to apply the bound (\ref{schurRange}) in Theorem \ref{mainResult}.\parspace

Let $D^{(n)}$ denote the set being enumerated in (\ref{Dn}). To see when $|D^{(n)}|$ stabilizes, we can construct connective maps that are eventually bijections. Define the map $\psi_n:D^{(n)}\rightarrow D^{(n)}$ as follows. Given a labeled Dyck path $P$ with underlying Dyck path $\pi$, extend the Dyck path $\pi$ to the Dyck path, $NE\pi$, on an $(n+1)$-by-$(n+1)$ grid. Keep the same labeling, while adding $n+1$ next to this newly created north step. See Fig \ref{modified labeled Dyck path} for an example.\parspace

The new reading word can be obtained from $\sigma(P)$ by adding an $n+1$ to the end of its one-line notation. Hence, $\psi_n(P)$ is an $\alpha^{(n+1)}$-shuffle whenever $P$ is an $\alpha^{(n)}$-shuffle. One can easily verify that $\text{area}(\pi)=\text{area}(NE\pi)$ and $\text{dinv}(\psi_n(P))=\text{dinv}(P)$ since there are no new area-contributing cells and $n+1$ does not form any diagonal inversion pairs. Hence, $\psi_n:D^{(n)}\rightarrow D^{(n+1)}$ is well-defined.\parspace

\begin{figure}[ht!]
    \centering
        \begin{tikzpicture}
    \draw[very thin, gray] (0,0) grid (7,7);
    
    \coordinate (start) at (0,0);

    \node at (0.5,0.5) {7};
    \node at (1.5,1.5) {2};
    \node at (1.5,2.5) {3};
    \node at (3.5,3.5) {1};
    \node at (3.5,4.5) {4};
    \node at (4.5,5.5) {5};
    \node at (6.5,6.5) {6};
    
    \draw[thick, red, -] 
        (start)
        -- ++(0,1)   
        -- ++(1,0)   
        -- ++(0,1)   
        -- ++(0,1)   
        -- ++(1,0)   
        -- ++(1,0)   
        -- ++(0,1)   
        -- ++(0,1)   
        -- ++(1,0)   
        -- ++(0,1)   
        -- ++(1,0)   
        -- ++(1,0)   
        -- ++(0,1)   
        -- ++(1,0) ;  
\end{tikzpicture}
    \caption{The labeled Dyck path obtained from the one in Fig \ref{labeled Dyck path} after applying $\psi_6$.}
    \label{modified labeled Dyck path}
\end{figure}
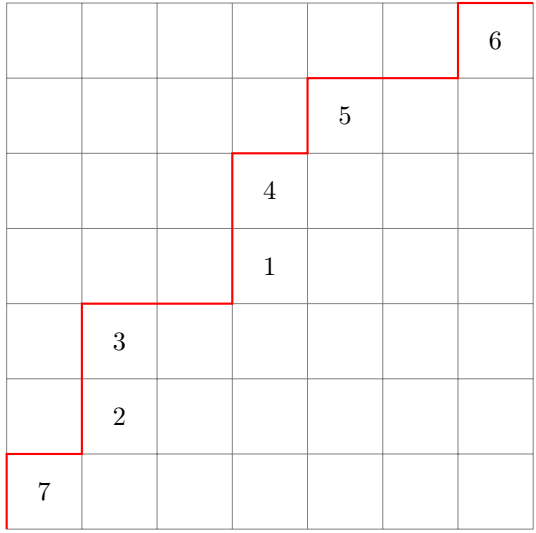

The map is clearly injective but not necessarily surjective for small $n$. We claim the map is surjective, once $n\geq i+j+|\mu|$ and $n\geq |\mu|+\mu_1$, i.e. that
$$\rg^m_{\mu}(\mathcal{F}([DR_n]_{(i,j)}))\leq |\mu|+\max(i+j,\mu_1).$$ Let $P\in D^{(n+1)}$. Since $\text{area}(P)=j$, there are at least $n+1-j$ rows with no area contributing cells. Hence there are at least $n+1-j$ cells on the diagonal $y=x$.\parspace

Now the reading word of $P$ must contain the values $|\mu|+1,\hdots,n,n+1$ in the same relative order, since the reading word is an $\alpha^{(n+1)}$-shuffle. Suppose for sake of contradiction that the bottom leftmost label of $P$ (i.e. the last entry of the reading word) is $a\neq n+1$. Then the relative order of $|\mu|+1,\hdots,n,n+1$ forces $a\leq |\mu|$. Looking for potential diagonal inversion pairs among labels in the diagonal $y=x$, the cell containing $a$ is southwest of at least $n-j$ other cells on the diagonal $y=x$. Since $a\leq |\mu|$, at most $|\mu|-1$ of these cells have a label less than $a$. Hence there are at least 
$$n-j-(|\mu|-1)>i=\text{dinv}(P)$$
diagonal inversion pairs, a contradiction. Hence the final entry of the reading word must be $n+1$.\parspace

Since labeled Dyck paths must be increasing up columns and $n+1$ is at the bottom of its column, $n+1$ is the only label in its column. Hence, the Dyck path of $P$ is of the form $NE\pi$ and $P$ lies in the image of $\psi_n$.\parspace

It follows that 
$$\left\langle \mathcal{F}\p{[DR_n]_{(i,j)}},h_{\mu[n]}\right\rangle$$
is stable once $n\geq i+j+|\mu|$ and so (\ref{schurRange}) implies that 

$$\left\langle \mathcal{F}\p{[DR_n]_{(i,j)}},s_{\lambda[n]}\right\rangle$$
is stable once 
$$n\geq \max\p{2|\lambda|,\max_{|\mu|\leq |\lambda|}(\max(i+j,\mu_1)+|\mu|)}=\max(2|\lambda|,\,i+j+|\lambda|)$$
\end{proof}

\begin{rem}
    The maps, $\psi_n$, defined in the proof above are injective for all $n$ and provide a combinatorial proof of inequality (\ref{monomIneq}) for $V_n=[DH_n]_{i,j}$
\end{rem}

Using Proposition \ref{coinvWeight}, we have $\text{weight}([DH_n]_{(i,j)})\leq i+j$. Note that it is in fact possible to show that $\text{weight}([DH_n]_{(i,j)})=i+j$ using combinatorial methods. Since $([DH_n]_{(i,j)})_n$ arises from a torsion-free \textbf{FI}-module and has monomial stable Frobenius images, it suffices to show that $\dim([DH_n]_{(i,j)})$ eventually agrees with a polynomial in $n$ of degree $i+j$, in order to conclude that the weight is $i+j$. This precise computation is carried out in Wang's dissertation \cite{wang}. Hence, (\ref{totalRange}) of Theorem \ref{mainResult} applies yielding the following corollary which agrees with the range given in Proposition \ref{coinvFarbRange}.
 \begin{cor}
      The $(i,j)$-bi-graded component of the diagonal harmonics $DH_n$ is URS with URS stable range $n\geq 2(i+j)$.
 \end{cor}
 
\section{Applications to Garsia-Haiman modules}\label{macSec}
In this section, we describe $S_n$-modules known as the Garsia-Haiman modules \cite{GHmodules}, whose bi-graded Frobenius images are $\tilde{H}_\mu$, the modified Macdonald polynomials. We will then show that they have an \textbf{FI}-module structure and hence are representation stable by the noetherian property of \textbf{FI}-modules. We will then provide an alternative direct approach to showing representation stability from the monomial expansions of their Frobenius images, while obtaining new insights into the combinatorics of Macdonald polynomials (see Corollary \ref{qtKostkaMonotonicity}).

\subsection{Garsia-Haiman modules}
As before, we take $DH_n$ to denote $H^{(2)}_n$, the diagonal harmonics, and denote the variable sets $X^{(1)}_n$ and $X^{(2)}_n$ instead by $x_1,\cdots,x_n$ and $y_1,\dots,y_n$.\parspace 

Given a partition shape $\nu\vdash n$ and a cell $c$ in $\nu$, we define the co-leg and co-arm of $c$, $l'(c)$ (resp. $a'(c)$), to be the number of cells below and in the same column as $c$ (resp. left and in the same row as $c$). Define the $\nu$-bialternant
$$\Delta_{\nu}=\det\p{x_i^{l'(c)}y_i^{a'(c)}}_{i=1,\hdots,|\nu|,~c\in \nu},$$ where $c$ ranges over some enumeration of the cells of $\nu$ and $\Delta_\nu$ is only defined up to sign since we don't provide an ordering on the cells.\parspace

For example, if $\nu=(3,2)$, then the (coleg,coarm) value pairs as we range over the cells are (0,0), (0,1), (0,2),(1,0), (1,1) and the $(3,2)$-bialternant is
$$\Delta_{(3,2)}=\det\begin{pmatrix}
    1&1&1&1&1\\
    x_1&x_2&x_3&x_4&x_5\\
    x_1y_1&x_2y_2&x_3y_3&x_4y_4&x_5y_5\\
    y_1&y_2&y_3&y_4&y_5\\
    y_1^2&y_2^2&y_3^2&y_4^2&y_5^2
\end{pmatrix}.$$

It is known (see solution to Exercise 2.12 \cite{haglund}) that $\Delta_\nu\in DH_n$ 
for all $n\geq |\nu|$. Let 
$$\mathcal{L}(f):=\text{Span}_\C(\text{partial derivatives of }f).$$
Since solutions to the differential equations associated with $DH_n$ are closed under partial differentiation (see Exercise 2.12 \cite{haglund}), we have
$$GH_{\nu}:=\mathcal{L}(\Delta_{\nu})\subseteq DH_n$$
is a multi-graded $S_n$-submodule of $DH_n$, which is referred to as the \textbf{Garsia-Haiman module} corresponding to $\nu$. We verify the following simple lemma.

\begin{lem}\label{MDcontainment}
    If $\mu\subset \nu$ are pair of partitions, then $GH_{\mu}\subset GH_\nu$.
\end{lem}

\begin{proof}
It suffices to show this for $\mu\subset \nu$ with $\mu\vdash n$ and $\nu\vdash n+1$, differing by a single cell. It is further enough to show that the bialternant $\Delta_{\mu}\in GH_{\nu}$, since $GH_{\nu}$ is closed under taking partial derivatives and the action of $S_n$. Let
    $$A=\p{x_i^{l'(c)}y_i^{a'(c)}}_{i=1,\dots,|\mu|,~c\in \mu^{(n)}},$$
    so that $\det(A)\underset{\sgn}{=}\Delta_{\mu}$, where $\underset{\sgn}{=}$ means equality up to sign. Let $\hat{c}$ be the unique cell $\nu/\mu$ so that
    $$\Delta_{\nu}\underset{\sgn}{=}\det\begin{bmatrix} & x_1^{l'(\hat{c})}y_1^{a'(\hat{c})}\\A & \vdots\\ & x_n^{l'(\hat{c})}y_n^{a'(\hat{c})}\\ (x^{l'(c)}_{n+1}y^{a'(c)}_{n+1})_{c\in \mu} & x_{n+1}^{l'(\hat{c})}y_{n+1}^{a'(\hat{c})}\end{bmatrix}\underset{\sgn}{=}x_{n+1}^{l'(\hat{c})}y_{n+1}^{a'(\hat{c})}\det(A)+\sum_{c\in\mu}x^{l'(c)}_{n+1}y^{a'(c)}_{n+1}f_c,$$
for some polynomials $f_c\in \C[x_1,\cdots,x_n,y_1,\cdots,y_n]$, independent of $x_{n+1},y_{n+1}$. Then
$$\p{\frac{\partial}{\partial x_{n+1}}}^{l'(\hat{c})}\p{\frac{\partial}{\partial y_{n+1}}}^{a'(\hat{c})}\Delta_{\nu}\underset{\sgn}{=}\p{l'(\hat{c})!\cdot a'(\hat{c})!}\Delta_{\mu}$$
since every cell, $c$, of $\mu$ is either strictly left of $\hat{c}$ (so that $a'(c)<a'(\hat{c})$) or strictly below $\hat{c}$ (so that $l'(c)<l'(\hat{c})$). This verifies that $\Delta_{\mu}\in GH_{\nu}$. 
\end{proof}

For the remainder of the section we will always let $\mu^{(\bullet)}$ denote a sequence of partitions $\mu^{(0)}\subset \mu^{(1)}\subset \dots\subset\mu^{(n)}\subset\dots$ such that $\mu^{(n)}\vdash n$. It follows that the sequence
$$GH_{\mu^{(0)}}\subset \dots \subset GH_{\mu^{(n)}}\subset GH_{\mu^{(n+1)}}\subset \dots$$
with inclusions as connective maps is a sub-\textbf{FI}-module of $(DH_n,i_n)_n$, which we will simply denote as $GH_{\mu^{(\bullet)}}$.
By Theorem 1.3 of \cite{FI-module}, each $(i,j)$-bi-graded component of $GH_{\mu^{(\bullet)}}$ is a finitely generated \textbf{FI}-module with weight $\leq i+j$. \parspace

\subsection{Application of monomial stability to modified Macdonald polynomials}
In this subsection, we will show that the $GH_{\mu^{(\bullet)}}$ are uniformally representation stable directly by looking at their Frobenius images. Since these modules are torsion-free, it suffices to show that they form a URMS sequence by Corollary \ref{URMS->URS}. We have the following Proposition of Haiman, which motivates the original interest in Garsia-Haiman modules as models of the more well-known modified Macdonald polynomials.

\begin{prop}[\cite{haimanMD}]\label{MDFrob}
    As an ungraded module, $GH_\nu$, is isomorphic to the regular representation of $S_{|\nu|}$. Letting $q$ track the $y$-weight and $t$ track the $x$-weight, we have that the bi-graded Frobenius image of the Garsia-Haiman module is
    $$\mathcal{F}(GH_\nu;q,t)=\tilde{H}_\nu[X;q,t],$$
    the modified Macdonald polynomials described in Definition \ref{modMacDef}. So in particular, $\tilde{H}_\nu[X;1,1]=h_{(1^{|\nu|})}$.
\end{prop}

As with the diagonal harmonics and $\nabla e_n$, there are no known combinatorial formulas for the Schur decompositions of the modified Macdonald polynomials. The coefficients for the Schur expansion of $\tilde{H}_\nu[X;q,t]$ are referred to as the \textbf{modified $q,t$-Kostka numbers}, $\tilde{K}_{\lambda,\nu}(q,t)\in \N[q,t]$, satisfying
    $$\tilde{H}_\nu[X;q,t]=\sum_{\lambda \vdash |\nu|}\tilde{K}_{\lambda,\nu}(q,t)\, s_\lambda[X].$$
So in particular, $\tilde{K}_{\lambda,\nu}(1,1)=K_{\lambda,\nu}$ reduces to an ordinary Kostka number by Proposition \ref{MDFrob}. We get the following combinatorial inequality as an immediate corollary of applying Corollary \ref{monomialMonotonicity} to the torsion-free \textbf{FI}-module $GH_{\mu^{(\bullet)}}$.

\begin{cor}\label{qtKostkaMonotonicity}
    Let $\mu^{(1)}\subset \mu^{(2)}\subset\cdots$ be a sequence of partitions where $\mu^{(n)}\vdash n$ and $\lambda$ is some partition. Then for $n\geq |\lambda|+\lambda_1$
    $$\tilde{K}_{\lambda[n+1],\mu^{(n+1)}}(q,t)-\tilde{K}_{\lambda[n],\mu^{(n)}}(q,t)\in \N[q,t]$$
    and for $n\geq |\eta|+\eta_1$
    $$\langle \tilde{H}_{\mu^{(n+1)}},h_{\eta[n+1]}\rangle-\langle \tilde{H}_{\mu^{(n)}},h_{\eta[n]}\rangle\in\N[q,t].$$\hfill $\square$
\end{cor}

With the knowledge that the the monomial coefficients $\langle \mathcal{F}([GH_{\mu^{(n)}}]_{i,j}),h_{\eta[n]}\rangle$ form a non-decreasing sequence of natural numbers, monomial stability will follow as soon as we can show that they are bounded from above. For this we can leverage the following miraculous combinatorial monomial formula provided by Haglund, Haiman, and Loehr \cite{HHL}.

\begin{defi}\label{modMacDef}
    The \textbf{modified Macdonald polynomial} for a partition $\nu\vdash n$, is a symmetric function in variable set $X$ over $\C(q,t)$:
    $$\tilde{H}_\nu[X;q,t]=\sum_{\sigma:\nu\rightarrow \N}q^{\text{inv}(\sigma)}t^{\text{maj}(\sigma)}x_\sigma,$$
    where the sum is taken over all $\N$-labelings of the cells of $\nu$, called fillings, the statistics inv and maj are described below, and $x_{\sigma}=\prod_{c\in \nu}x_{\sigma(c)}$.
\end{defi}

We thus get the following formula for the monomial coefficients of the bi-graded components of the Garsia-Haiman modules
    \begin{align}
        \langle \mathcal{F}([GH_{\mu^{(n)}}]_{i,j}),h_{\eta[n]}\rangle&=\langle [q^it^j]\tilde{H}_{\mu^{(n)}},h_{\eta[n]}\rangle\\
        &=\#\{\sigma:\mu^{(n)}\rightarrow \N:\text{content }0^{n-|\eta|},1^{\eta_1},\hdots,l^{\eta_l},~\text{inv}(\sigma)=i,~\text{maj}(\sigma)=j\}\label{fillingSet}
    \end{align}
    where $l=l(\eta)$.\parspace

The \textit{reading order} of a diagram $\nu$ is the order on cells that goes left-to-right along rows starting from the top row and working down. The \textit{standardization} of a labeling $\sigma:\nu\rightarrow \N$ with content $\alpha=(\alpha_0,\alpha_1,\cdots)$, denoted by $\text{st}(\sigma)$, is an injective labeling $\sigma:\nu\rightarrow \{1,\cdots,|\nu|\}$ that labels the 0's appearing in $\nu$ as $1,2,\cdots,\alpha_0$ in the order that they appear in the reading order. The 1's appearing in $\nu$ are labeled $\alpha_0+1,\cdots,\alpha_0+\alpha_1$ in reading order and so on. A \textit{triple} consists of cells with (row,column) coordinates of the form $u=(r,c)$, $v=(r-1,c)$ and $w=(r,c+k)$ for $r,c,k\geq 1$, where $u,w$ must be cells of $\nu$, but $v$ is allowed to be a cell immediately below $\nu$. Extend $\text{st}(\sigma)$ to cells immediately below $\nu$ by labeling them $-\infty$. A triple is an \textit{inversion triple} of $\sigma$ if the labels $\text{st}(\sigma)(u),\text{st}(\sigma)(w),\text{st}(\sigma)(v)$ are decreasing clockwise. The statistic $\text{inv}(\sigma)$ counts the number of inversion triples of $\sigma$.\parspace

A \textit{descent} of $\sigma$ is a cell $u=(r+1,c)$ in $\nu$, such that $v=(r,c)$ is also in $\nu$ and $\sigma(u)>\sigma(v)$. Letting $\text{Des}(\sigma)$ denote that set of descents we define the \textbf{major index} statistic
$$\text{maj}(\sigma):=\sum_{u\in \text{Des}(\sigma)}(\text{leg}(u)+1),$$
where leg$(u)$ is the number of cells of $\nu$ in the same column as $u$ and above $u$. See Fig \ref{macdonaldEx} for an example.

\begin{figure}[ht!]
    \centering
    \begin{ytableau}
    1&4&3\\
    0&2&1\\
    0&3&0&1
    \end{ytableau}
    \caption{An $\N$-labeling of $\nu=(4,3,3)$. The descents occur at cells $(2,3),(3,1),(3,2),(3,3)$ yielding $\text{maj}(\sigma)=2+1+1+1=5$. The inversion triples are formed by $((3,2),(3,1),(3,3))$, $((1,2),(1,3))$, and $((1,2),(1,4))$ yielding $\text{inv}(\sigma)=3$.}
    \label{macdonaldEx}
\end{figure}

\begin{ques}\label{monotonicOpenProblem}
    The formula for the monomial coefficient (\ref{fillingSet})
    begs the question whether or not Corollary \ref{qtKostkaMonotonicity} (which followed from the non-trivial representation theoretic Lemma \ref{monotonictyLemma}) can be proven directly. In other, words can one describe an injective map from
    $$\{\sigma:\mu^{(n)}\rightarrow \N:\text{content }0^{n-|\eta|},1^{\eta_1},\dots,l^{\eta_l},~\text{inv}(\sigma)=i,~\text{maj}(\sigma)=j\}$$
    into
    $$\{\sigma':\mu^{(n+1)}\rightarrow \N:\text{content }0^{n+1-|\eta|},1^{\eta_1},\dots,l^{\eta_l},~\text{inv}(\sigma')=i,~\text{maj}(\sigma')=j\}.$$
    Note that this may be an intricate problem, since fillings can end up looking quite different in order to meet the same inv and maj target values
    in the new shape $\mu^{(n+1)}$.
\end{ques}

\begin{lem}\label{upperBoundMonomial}
        Fix $i,j\geq 0$ and a partition $\eta$. We have the following upper bound
        \begin{equation}\label{upperBoundEq}
            [q^it^j]\langle \tilde{H}_{\mu},h_{\eta[n]}\rangle \leq \binom{k}{\eta[k]},
        \end{equation}
        where $n$ and $\mu\vdash n$ are arbitrary and $k=(|\eta|+i)(|\eta|+j)$. We interpret the left hand side to be zero when $\eta[n]$ is not defined and interpret $\binom{k}{\eta[k]}=1$ in the single case when $\eta[k]$ is not defined ($i,j=0$ and $\eta=(1)$).
\end{lem}
\begin{proof}
    We will show the upper bound holds for any given $n$ and $\mu\vdash n$. Let $A\subseteq_{\text{diag}}{\mu}$ be the subset of the cells of $\mu$ consisting of the rightmost $|\eta|+i$ cells of the bottom row and the leftmost $|\eta|+i$ cells of any higher row. Let $B\subseteq_{\text{diag}} \mu$ be the subset of cells which are in either the topmost $j$ or bottommost $|\eta|$ cells of a column. Let $Z=A\cap B$. Note that by construction $|Z|\leq (|\eta|+i)(|\eta|+j)$. See Fig \ref{zoneExample} for an example.\parspace
    
    \begin{figure}[ht!]
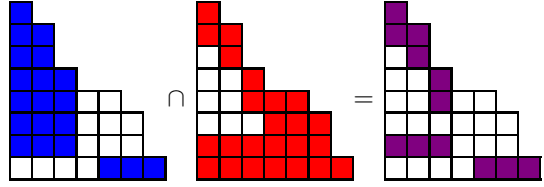

        \centering
        \ytableausetup{smalltableaux}
        \begin{ytableau}
            *(blue)\\
            *(blue)& *(blue)\\
            *(blue)& *(blue)\\
            *(blue)&*(blue)& *(blue)\\
            *(blue)&*(blue)&*(blue)&&   &\none&\none&\none[\cap]\\
            *(blue)&*(blue)&*(blue)&&&\\
            *(blue)&*(blue)&*(blue)&&&\\
            &&&&*(blue)&*(blue)&*(blue)
        \end{ytableau}
        \begin{ytableau}
            *(red)\\
            *(red)& *(red)\\
            & *(red)\\
            && *(red)\\
            && *(red)&*(red)&*(red) &\none&\none&\none[=]\\
            &&&*(red)&*(red)&*(red)\\
            *(red)&*(red)&*(red)&*(red)&*(red)&*(red)\\
            *(red)&*(red)&*(red)&*(red)&*(red)&*(red)&*(red)
        \end{ytableau}
        \begin{ytableau}
            *(violet)\\
            *(violet)&*(violet)\\
            &*(violet)\\
            &&*(violet)\\
            &&*(violet)&&\\
            &&&&&\\
            *(violet)&*(violet)&*(violet)&&&\\
            &&&&*(violet)&*(violet)&*(violet)
        \end{ytableau}
        \caption{An example of the zone $Z$ for $\mu=(7,6,6,5,3,2,2,1)$, $|\eta|=2$, $i=1$ and $j=2$, with the set $A$ in blue, the set $B$ in red, and the set $Z$ in purple.}
        \label{zoneExample}
    \end{figure}

We claim that any filling $\sigma:\mu\rightarrow \N$ contributing to the count $[q^it^j]\langle \tilde{H}_{\mu},h_{\eta[n]}\rangle$, i.e., a filling with content $0^{n-|\eta|}, 1^{\eta_1},\dots,l^{\eta_l}$ and $\inv(\sigma)=i$ and $\maj(\sigma)=j$, must have all its non-zero entries in $Z$.\parspace

Suppose for contradiction that $\sigma$ attached a non-zero label to a cell, $c$, outside of $A$. If $c$ is in the first row, then there are at least $|\eta|+i$ cells right of $c$, with at most $|\eta|-1$ of them being non-zero. But this means that $c$ creates at least $i+1$ inversions with the cells labeled zero to its right, a contradiction. Similarly, if $c$ is not in the bottom row then it has at least $|\eta|+i$ columns to its left, of which at least $i+1$ contain only zeros. But then $c$ forms an inversion triple of the form \begin{ytableau}
    0&\none&c\\
    0
\end{ytableau} with each of those columns, yielding the same contradiction. Hence $c\in A$.\parspace

Now we will show that $c\in B$. If $c$ is in the top $j$ entries of its column then we are done.
Suppose otherwise and observe that $c$ cannot form a descent, or else it would contribute at least $j+1$ to $\maj(\sigma)$. Hence $c$ is in the bottom row or the cell immediately below $c$, $c'$, has a non-zero label. But $c'$ cannot form a descent by the same reasoning and we can repeat the argument to conclude that there is a contiguous vertical interval of cells starting from $c$ to the bottom row containing non-zero labels. Hence $c$ must be in the bottom $|\eta|$ rows of $\mu$, implying $c\in B$. Thus, we have shown that if $\sigma(c)>0$, then $c\in A\cap B=Z$.\parspace

Since $[q^it^j]\langle \tilde{H}_{\mu},h_{\eta[n]}\rangle$ is bounded by the total number of ways to fill $\sigma|_Z:Z\rightarrow \N$ using content $0^{|Z|-|\eta|},1^{\eta_1},\dots,l^{\eta_l}$, we have
$$[q^it^j]\langle \tilde{H}_{\mu},h_{\eta[n]}\rangle\leq \binom{|Z|}{|Z|-|\eta|,\eta_1,\dots,\eta_l}\leq \binom{k}{\eta[k]},$$
where $k=(|\eta|+i)(|\eta|+j)$.
\end{proof}

Hence, we arrive at the following proposition.

\begin{prop}[Direct proof of $(GH_{\mu^{(\bullet)}})$ stability]
    Each $(i,j)$-bi-graded component of $(GH_{\mu^{(\bullet)}})$ is uniformally representation stable.
\end{prop}

\begin{proof}
    This follows by Theorem \ref{mainResultTorFree}, the fact that $\wt(GH_{\mu^{(\bullet)}})\leq i+j$, and the upper bounds on monomial coefficients established in Lemma \ref{upperBoundMonomial}.
\end{proof}

\subsection{Special case: $\mu^{(n)}=\mu[n]$}

In this subsection, we focus on the special case where we fix some $\mu$ and explicitly compute a stable range for the sequence of $(GH_{\mu[n]})_n$ for $n\geq |\mu|+\mu_1$.

\begin{lem}
    The monomial coefficient $\langle [q^it^j]\tilde{H}_{\mu[n]},h_{\eta[n]}\rangle$ stabilizes once $n\geq \max(|\mu|+\mu_1+|\eta|+i,~|\eta|+\eta_1)$.
\end{lem}

\begin{proof}
Fix partitions $\mu,\nu$ and $i,j$ and let $M^{(n)}$ denote the set of being enumerated in (\ref{fillingSet}). Then $$\langle [q^it^j]\tilde{H}_{\mu[n]},h_{\eta[n]}\rangle=|M^{(n)}|$$ and we wish to show $|M^{(n)}|$ stabilizes.\parspace

Suppose $n\geq |\mu|+\mu_1+|\eta|+i$ and $n\geq |\eta|+\eta_1$ so that $\mu[n]$ and $\eta[n]$ are well-defined. Define the map $\gamma_n:M^{(n)}\rightarrow M^{(n+1)}$ as follows. For $\sigma \in M^{(n)}$, shift the bottom row of the corresponding diagram to the right by one cell and insert a new cell labeled 0 in the bottom left corner. So for instance:

$$\gamma_n:\begin{ytableau}
    1&0&2\\0&0&4&3&1&0
\end{ytableau}\,\longmapsto\, 
 \begin{ytableau}
    1&0&2\\0&0&0&4&3&1&0
\end{ytableau}$$

We have to check that the map is well-defined and surjective. We claim that the first $\mu_1+1$ entries of the bottom row of any $\sigma'\in M^{(n+1)}$ are 0. Suppose for contradiction that there was a non-zero value $a$ among these first $\mu_1+1$ cells. Then there are at least $n+1-|\mu|-(\mu_1+1)=n-|\mu|-\mu_1$ entries right of $a$. Of these, at most $|\eta|-1$ are allowed to be non-zero (since $a$ is one of the $|\eta|$ non-zero values in the labeling). Hence $a$ would contribute at least $n-|\mu|-\mu_1-(|\eta|-1)>i$ inversion triples formed by taking $a$, the phantom cell labeled $-\infty$ below it, and a cell labeled 0 to the right of it. This contradicts that $\text{inv}(\sigma')=i$ for $\sigma'\in M^{(n+1)}$.\parspace

This implies that if we let $\sigma$ be obtained from a $\sigma'\in M^{(n+1)}$ by deleting the bottom left cell and shifting the bottom row to the left by one, then $\sigma\in M^{(n)}$ and $\gamma_n(\sigma)=\sigma'$, which implies that $\gamma_n$ is well-defined and bijective. Indeed, the first $\mu_1$ cells of the bottom row of $\sigma$ will be labeled 0 and so the status of descents and inversion triples does not change when going from $\sigma'$ to $\sigma$. Thus, $\sigma'$ and $\sigma$ have the same inv and maj values.
\end{proof}

\begin{cor}\label{qtKostkaRange}
    The coefficient of $\mathbb{S}^{\lambda[n]}$ in $[GH_{\mu[n]}]_{(i,j)}$ stabilizes once $n\geq \max(2|\lambda|,|\lambda|+|\mu|+\mu_1+i)$. Thus, $[GH_{\mu[n]}]_{(i,j)}$ is URS with URS stable range $n\geq \max(2(i+j),|\mu|+\mu_1+2i+j)$.
\end{cor}
\begin{proof}
    By Theorem \ref{mainResult}, the coefficient of $\mathbb{S}^{\lambda[n]}$ in $[GH_{\mu[n]}]_{(i,j)}$ stabilizes once $n\geq 2|\lambda|$ and $ \max(|\mu|+\mu_1+|\eta|+i,~|\eta|+\eta_1)$ for all $|\eta|\leq |\lambda|$. The maximum of all these quantities is $\max(2|\lambda|,|\lambda|+|\mu|+\mu_1+i))$. Since the weight of $[GH_{\mu[n]}]_{(i,j)}$ is $\leq i+j$, we get the stable range in the statement.
\end{proof}

We may restate Corollary \ref{qtKostkaRange} as follows. The coefficient $[q^it^j]\tilde{K}_{\lambda[n],\mu[n]}(q,t)$
is non-zero only if $|\lambda|\leq i+j$ and
stabilizes once $n\geq  \max(2|\lambda|,|\lambda|+|\mu|+\mu_1+i)$.
  
\begin{ex}\label{kostkaEx}
     One case in which the $q,t$-Kostka number is well-understood
     occurs when $\lambda[n]$ is a hook shape. Let $B_{\mu[n]}=\sum_{c\in \mu[n]}q^{a'(c)}t^{l'(c)}$ so that in particular $B_{\mu[n]}-1=\sum_{c\in \mu[n]\backslash (1,1)}q^{a'(c)}t^{l'(c)}$ is the sum over all but the bottom left cell. From \cite{macdonald1998symmetric}, we have the plethystic identity
$$\tilde{K}_{1^k[n],\mu[n]}(q,t)=\tilde{K}_{(n-k,1^k),\mu[n]}(q,t)=e_k[B_{\mu[n]}-1],$$
the elementary symmetric function $e_k$ with variables set to $q^{a'(c)}t^{l'(c)}$ for $c\in \mu[n]\backslash (1,1)$. We expect the coefficient of $q^it^j$ to stabilize once $n\geq \max(2k,k+|\mu|+\mu_1+i)$. It in fact stabilizes earlier; once $n\geq |\mu|+i+1$. This is because the additional term contributed by the cell added to the bottom row $B_{\mu[n+1]}-B_{\mu[n]}=q^{n+1-|\mu|-1}t^0$, which cannot contribute to the coefficient of $q^it^j$, once $n-|\mu|>i$.
\end{ex}

As noted by the authors in \cite{Garsia1999}, the formula in Example \ref{kostkaEx} admits the following plethystic formula generalization to all shapes $\lambda[n]$:
\begin{equation}\label{plethForm}
    \tilde{K}_{\lambda[n],\mu[n]}=\nabla^{-1} s_\lambda \left[\frac{1-\epsilon(MtB_\mu-1)}{M}-1-\epsilon[n-|\mu|]_q\right].
\end{equation}
Since $\nabla^{-1}$ introduces negative powers of $q$ and $t$, one cannot simply repeat the argument in Example \ref{kostkaEx}. However, with more control on the resulting positive and negative powers it may be possible to get a similar stable range to the one obtained above.
\parspace

Alternatively, in \cite{garsiaPleth} a formula of the form
\begin{equation}\label{k_lambda}
    \tilde{K}_{\lambda[n],\mu[n]}=k_{\lambda}[B_{\mu[n]};q,t],
\end{equation}
is found where $k_{\lambda}[X;q,t]$ is an inhomogeneous symmetric function with highest degree $|\lambda|$ and coefficients being Laurent polynomials in $q,t$. In Example \ref{kostkaEx}, $k_{1^k}[X]=e_k[X-1]$. Again, the coefficients have negative powers so the argument cannot run through as in the example.

\begin{ques}
    Can one alternatively prove Corollary \ref{qtKostkaRange} using a careful analysis of equation (\ref{plethForm}) or (\ref{k_lambda})?
\end{ques}

\section{Further work}

Similar to Wang's work \cite{wang} on computing the polynomial that eventually agrees with $\dim([DH_{n}]_{(i,j)})$, we can ask about the polynomial corresponding to $\dim([GH_{\mu[n]}]_{(i,j)})$ or more generally, $\dim([GH_{\mu^{(n)}}]_{(i,j)})$. In particular, we know that the degree of the polynomial will agree with $\wt([GH_{\mu^{(\bullet)}}]_{(i,j)})$ and it is interesting to see whether or not this degree attains its maximum possible value, $i+j$, like in the case of $DH_n$. We are currently in the process of carrying out this computation using the expansion of $\tilde{H}_\mu$ in the symmetric functions of Lascoux, Leclerc, and Thibon.\parspace

A related object of interest, is the intersection of Garsia-Haiman modules, $GH_{\mu}$ for varying $\mu$, although proofs of the conjectured Frobenius image and even the dimension of the intersection, remain elusive (see \cite{sciFi}). Recently, a conjectural monomial formula was presented for the intersection $GH_{\mu}\cap GH_{\lambda}$, where $\mu,\lambda$ differ by a cell \cite{butlerConj}. It may be interesting to explicitly observe the monotonicity described by Corollary \ref{monomialMonotonicity}, as well as the stability of the monomial coefficients as we grow $\mu,\lambda$ as this would provide further evidence for the validity of this formula.\parspace

We saw how for torsion-free \textbf{FI}-modules the eventual polynomial dimension of $V_n$ can provide a bound on weight. We ask if one can make similar conclusions for non-torsion-free \textbf{FI}-modules, provided we know something about the character values.

\begin{ques}
    Suppose that for large $n$, $\dim(V_n)=p(n)$ for a degree $K$ polynomial $p$ and the character values 
    $$\chi^{V_n}(\sigma\cdot (k+1~~~k+2~\cdots ~n))$$
    stabilize for all $\sigma \in S_k$ and all $k\leq K$. Is $(V_n)_n$ URMS?
\end{ques}

Note how this relates to polynomiality of characters: this provides a converse to Theorem 3.3.4 in \cite{FI-module}, while checking less conditions about the character values i.e. not demanding that $\chi^{V_n}(\sigma)$ is a polynomial in cycle counts of $\sigma$.\parspace

In a different direction, we can also consider combinatorial Hopf algebras other than the symmetric functions. There is an analogue of the Frobenius characteristic map between the Grothendieck ring of the 0-Hecke algebras, $(H_n(0))_n$, and the quasi-symmetric functions (see \cite{Laudone} or \cite{huang} for instance). Laudone \cite{Laudone} introduces the notion of representation stability for sequences, $(V_n)_n$, where $V_n$ is an $H_n(0)$-module. Instead of having the coefficient of $s_{\lambda[n]}$ stabilizing, the quasi-symmetric Frobenius images have the multiplicity of certain $(P,\omega)$-partition generating functions stabilize (see \cite{stanley2}).

\begin{ques}
    Can one generalize the methods developed for transferring stability between symmetric function bases to transferring stability between quasi-symmetric bases? Can these methods be applied to conclude 0-Hecke representation stability of sequences in the sense of \cite{Laudone}?
\end{ques}

\section*{Appendix}

As promised this appendix provides an alternate proof of Theorem  \ref{mainResult}(a) and in fact improve the ranges (\ref{m->s individual range}) and (\ref{m->s uniform range}).\parspace

\begin{prop}\label{altProp}
A sequence of $S_n$-modules $V=(V_n\in \textbf{Mod}_{S_n})_n$ is URMS if $\text{wt}(V)<\infty$ and the sequence of Frobenius images $\mathcal{F}(V_\bullet)$ is $m$-stable. In this case, the multiplicity of $\mathbb{S}^{\lambda[n]}$ in $V_n$ stabilizes once
   \begin{equation}\label{schurRange improve}
       n\geq \max_{|\mu|\leq |\lambda|}\rg^m_\mu(\mathcal{F}(V_\bullet)).
   \end{equation}
    and a uniform stable range of $V_\bullet$ is given by
    \begin{equation}\label{totalRange improve}
        n\geq \max_{|\mu|\leq \text{wt}(V)}\rg^m_\mu(\mathcal{F}(V_\bullet)).
    \end{equation}    
\end{prop}

Note that because we need to consider all $|\mu|\leq |\lambda|$ in (\ref{schurRange improve}), we can take $\mu=(|\lambda|)$ which leave $\mu[n]$ undefined for $n<2|\lambda|$. In all the applications it is possible to have a non-zero monomial coefficient corresponding to $\mu[n]$ for $n\geq 2|\lambda|$ and hence we anyways need $n\geq 2|\lambda|$ for (\ref{schurRange improve}) to hold.\parspace

The proof of Proposition \ref{altProp} does not rely on the general ideas of stability transfer and instead uses very specific properties of the Kostka numbers as seen in the following key lemma. Recall, that a symmetric function is Schur-positive if the coefficients of its Schur expansion are non-negative integers. A degree-$n$ homogeneous symmetric function is the Frobenius image of some $S_n$-representation if and only if it is Schur-positive.

\begin{lem}\label{kvec lemma}
    Fix $n$ and let $F_{n+i}=\sum_{\lambda}c_{\lambda,n+i}s_{\lambda[n+i]}$, for $i=0,1$ be two Schur-positive functions. Let $F_{n+i}=\sum_\mu d_{\mu,n+i}m_{\mu[n+i]}$ be the corresponding monomial expansions, with the usual assumption that $c_{\lambda,n+i}=0$ and $d_{\lambda,n+i}=0$ when $\lambda[n+i]$ is not defined. If $d_{\mu,n}=d_{\mu,n+1}$ for all $|\mu|\leq a$ for some fixed $a$, then $c_{\lambda,n}=c_{\lambda,n+1}$ for all $|\lambda|\leq a$.
\end{lem}

\begin{proof}
    Observe that we have the relation
    $$d_{\mu,n+i}=\sum_{\lambda}c_{\lambda,n+i}K_{\lambda[n+i],\mu[n+i]},$$
    where the terms of the sum are only non-zero when $|\lambda|\leq |\mu|$ by Lemma \ref{kostkaStable}. For $|\lambda|\leq a$, define the Kostka column vector, $\vec{K}_{\lambda,n}$ to be the vector $(K_{\lambda[n],\mu[n]})_{|\mu|\leq a}$. Then the condition $d_{\mu,n}=d_{\mu,n+1}$ for all $|\mu|\leq a$ becomes
    \begin{equation}\label{Kvec equality}
        \sum_{|\lambda|\leq a}c_{\lambda,n}\vec{K}_{\lambda,n}=\sum_{|\lambda|\leq a}c_{\lambda,n+1}\vec{K}_{\lambda,n+1}
    \end{equation}
    For $i=0,1$, let $T_i$ be the set of $\lambda$ with $|\lambda|\leq a$ and $c_{\lambda,n+i}\neq 0$. We claim that for all $\lambda\in T_1$, $\vec{K}_{\lambda,n}=\vec{K}_{\lambda,n+1}$.\parspace 
    
    Observe by the injectivity of the map, $\kappa_{n}$, defined in Lemma \ref{kostkaStable}, that $K_{\lambda[n],\mu[n]}\leq K_{\lambda[n+1],\mu[n+1]}$. We claim that for $|\lambda|\leq a$,  $c_{\lambda,n+1}=0$ whenever there exists a $|\mu|\leq a$ with $K_{\lambda[n],\mu[n]}< K_{\lambda[n+1],\mu[n+1]}$, which would imply that $\vec{K}_{\lambda,n}=\vec{K}_{\lambda,n+1}$ for $\lambda\in T_1$.\parspace
    
    If the inequality of Kostka number is strict, then
    $$\kappa_{n}:\text{SSYT}(\lambda[n],\mu[n])\rightarrow \text{SSYT}(\lambda[n+1],\mu[n+1])$$
    fails to be surjective and so there is a certain $T\in \text{SSYT}(\lambda[n+1],\mu[n+1])$ not in the image of $\kappa_n$. Note that it must be the case that the number of cells labeled 1 in $T$ is $n+1-|\mu|\leq \lambda_1$, i.e. $T$ has a cell above every cell containing a 1 or else $T$ would be in the image of the map. We also have that $|\mu|\geq |\lambda|$ since $\text{SSYT}(\lambda[n+1],\mu[n+1])$ is non-empty. Hence, $|\mu|\geq n+1-|\mu|$ and $2|\mu|-(n+1)\geq 0$. It follows then that $\text{SSYT}(\lambda[n+1],\tilde{\mu}[n+1])\neq \emptyset$ for $\tilde{\mu}=(n+1-|\mu|,1^{2|\mu|-(n+1)})$ since we can construct a $\tilde{T}\in \text{SSYT}(\lambda[n+1],\tilde{\mu}[n+1])$ by placing $n+1-|\mu|$ cells labeled 1 at the start of the bottom row, then placing $n+1-|\mu|$ cells labeled 2 at the start of the second row and filling in the remaining cells of $\lambda[n+1]$ with entries $3,4,\hdots$ in some order that creates a semi-standard Young tableau.\parspace
    
    Hence, $K_{\lambda[n+1],\tilde{\mu}[n+1]}\neq 0$, but $\tilde{\mu}[n]$ is not well-defined since $n<n+1=|\tilde{\mu}|+\tilde{\mu}_1$. Since $|\tilde{\mu}|=|\mu|\leq a$, we must have equality of
    $$d_{\tilde{\mu},n}=\sum_{\nu}c_{\nu,n}K_{\nu[n],\tilde{\mu}[n]}=0$$
    and $$d_{\tilde{\mu},n+1}=\sum_{\nu}c_{\nu,n+1}K_{\nu[n+1],\tilde{\mu}[n+1]}.$$
    By Schur-positivity, all the $c_{\nu,n+1}\geq 0$ and all the $K_{\nu[n+1],\tilde{\mu}[n+1]}\geq 0$ with $K_{\lambda[n+1],\tilde{\mu}[n+1]}>0$ and so it follows that $c_{\lambda,n+1}=0$.\parspace

    Thus, Eq (\ref{Kvec equality}) becomes
    $$\sum_{\lambda \in T_0\backslash T_1}c_{\lambda,n}\vec{K}_{\lambda,n}+\sum_{\lambda \in T_1\cap T_0}c_{\lambda,n}\vec{K}_{\lambda,n}=\sum_{\lambda \in T_1}c_{\lambda,n+1}\vec{K}_{\lambda,n}.$$
    We claim that the set of vectors, $\{\vec{K}_{\lambda,n}:\lambda\in T_0\cup T_1\}$, is linearly independent so that the above equation implies $T_0=T_1$ and $c_{\lambda,n}=c_{\lambda,n+1}$ for $|\lambda|\leq a$ concluding the proof of the lemma.\parspace

    For $\lambda\in T_0$, we have $\lambda[n]$ defined since $c_{\lambda,n}\neq 0$. For $\lambda\in T_1$, we have $\lambda[n+1]$ defined and so $\vec{K}_{\lambda,n}=\vec{K}_{\lambda,n+1}$ implies $\lambda[n]$ is defined as well. Fix a linear extension, $\succcurlyeq$, of the dominance order on the set of partitions $\{\mu[n]:|\mu|\leq a\}$. With respect to this order, every $\vec{K}_{\lambda,n}$ has $\mu$ entry $(\vec{K}_{\lambda,n})_\mu=K_{\lambda[n],\mu[n]}=1$ for $\mu=\lambda$ and $(\vec{K}_{\lambda,n})_\mu=0$ for all $\mu\succcurlyeq \lambda$. Hence the vectors must be linearly independent.
\end{proof}    

Proposition \ref{altProp} easily follows by taking Schur-positive $F_n=\mathcal{F}(V_n)$ since the lemma implies that once the monomial coefficients, $d_{\mu,n}$, stabilize for $|\mu|\leq |\lambda|$ (condition (\ref{schurRange improve})), we have stability of the Schur coefficients $c_{\lambda,n}$.

\bibliography{bibliography}

\end{document}